\documentclass[11pt, a4paper]{amsart}
\usepackage{a4wide}
\usepackage{graphicx,enumitem}
\usepackage{amsmath,amssymb}
\usepackage[english]{babel}
\usepackage{subfigure}
\usepackage{units}
\usepackage{diagbox}
\usepackage{xcolor}
\usepackage{caption}
\usepackage{ulem}
\usepackage{cancel}
\usepackage{scalerel,stackengine}

\usepackage{amssymb}
\DeclareSymbolFontAlphabet{\amsmathbb}{AMSb}%
\usepackage{mathbbol}
\usepackage{latexsym}

\usepackage[pdftex,
pdftitle={Regularity of integral operators},
bookmarksopen,
colorlinks,
linkcolor=black,
urlcolor=black,
citecolor=black
]{hyperref}
\hypersetup{pdfauthor={M. Kov\'acs, A. Lang and A. Petersson}}

\usepackage{dsfont}

\usepackage{mathtools}

\newcommand{\IP}{\amsmathbb{P}}
\newcommand{\R}{\amsmathbb{R}}

\newcommand{\C}{\amsmathbb{C}}

\newcommand{\N}{\amsmathbb{N}}

\newcommand{\cB}{\mathcal{B}}
\newcommand{\cC}{\mathcal{C}}
\newcommand{\cD}{\mathcal{D}} 
\newcommand{\cE}{\mathcal{E}}
\newcommand{\cF}{\mathcal{F}}
\newcommand{\cG}{\mathcal{G}}

\newcommand{\cL}{\mathcal{L}}

\DeclareMathOperator{\E}{\amsmathbb{E}} %
\DeclareMathOperator{\Cov}{\mathsf{Cov}}

\DeclareMathOperator{\trace}{Tr}

\DeclareMathOperator*{\esssup}{ess\,sup}

\newcommand{\dd}{\,\mathrm{d}}
\newcommand{\dom}{\mathrm{dom}} %

\newcommand{\inpro}[3][{}]{ \langle #2 , #3 \rangle_{#1} }
\newcommand{\norm}[2]{\| #1 \|_{#2}}

\newcommand{\Bignorm}[2]{\Big\| #1 \Big\|_{#2}}

\stackMath
\newcommand\reallywidehat[1]{%
	\savestack{\tmpbox}{\stretchto{%
			\scaleto{%
				\scalerel*[\widthof{\ensuremath{#1}}]{\kern-.6pt\bigwedge\kern-.6pt}%
				{\rule[-\textheight/2]{1ex}{\textheight}}%
			}{\textheight}%
		}{0.5ex}}%
	\stackon[1pt]{#1}{\tmpbox}%
}
\newtheorem{lemma}{Lemma}[section]

\newtheorem{theorem}[lemma]{Theorem}

\newtheorem{corollary}[lemma]{Corollary}
\theoremstyle{remark}
\newtheorem{remark}[lemma]{Remark}

\theoremstyle{definition}

\newtheorem{example}[lemma]{Example}

\begin{document}
	\title[Regularity of integral operators]{Hilbert--Schmidt regularity of symmetric integral operators on bounded domains with applications to SPDE approximations}
	
	\author[M.~Kov\'acs]{Mih\'aly Kov\'acs} \address[Mih\'aly Kov\'acs]{\newline Faculty of Information Technology and Bionics
		\newline P\'azm\'any P\'eter Catholic University
		\newline H-1444 Budapest, P.O. Box 278, Hungary.
		\newline and
		\newline Department of Differential Equations, Faculty of Natural Sciences, Budapest University of Technology and Economics
		\newline M\H{u}egyetem rkp. 3.
		\newline H-1111 Budapest, Hungary
		\newline and
		\newline Department of Mathematical Sciences
		\newline Chalmers University of Technology \& University of Gothenburg
		\newline S--412 96 G\"oteborg, Sweden.} \email[]{kovacs.mihaly@itk.ppke.hu}
	
	\author[A.~Lang]{Annika Lang} \address[Annika Lang]{\newline Department of Mathematical Sciences
		\newline Chalmers University of Technology \& University of Gothenburg
		\newline S--412 96 G\"oteborg, Sweden.} \email[]{annika.lang@chalmers.se}
	
	\author[A.~Petersson]{Andreas Petersson} \address[Andreas Petersson]{\newline The Faculty of Mathematics and Natural Sciences
		\newline Department of Mathematics
		\newline Postboks 1053, Blindern
		\newline 0316 Oslo, Norway.} \email[]{andreep@math.uio.no}
	
	\thanks{M.\ Kov\'acs acknowledges the support of the Marsden Fund of the Royal Society of New Zealand	through grant. no. 18-UOO-143, the Swedish Research Council (VR) through project no.\ 2017-04274 and the NKFIH through grant numbers 131545 and TKP2021-NVA-02. The work of A.\ Lang was partially supported by the Swedish Research Council (VR) (project no.\ 2020-04170), by the Wallenberg AI, Autonomous Systems and Software Program (WASP) funded by the Knut and Alice Wallenberg Foundation, and by the Chalmers AI Research Centre (CHAIR). The work of A. Petersson was supported in part by the Research Council of Norway (RCN) through project no.\ 274410, the Swedish Research Council (VR) through reg.~no.~621-2014-3995 and the Knut and Alice Wallenberg foundation. We are grateful to Prof.\ Giulia Di Nunno and two anonymous reviewers whose comments helped improve the manuscript. 
	}
	
	
	\begin{abstract}
		Regularity estimates for an integral operator with a symmetric continuous kernel on a convex bounded domain are derived. The covariance of a mean-square continuous random field on the domain is an example of such an operator. The estimates are of the form of Hilbert--Schmidt norms of the integral operator and its square root, composed with fractional powers of an elliptic operator equipped with homogeneous boundary conditions of either Dirichlet or Neumann type. These types of estimates, which couple the regularity of the driving noise with the properties of the differential operator, have important implications for stochastic partial differential equations on bounded domains as well as their numerical approximations. The main tools used to derive the estimates are properties of reproducing kernel Hilbert spaces of functions on bounded domains along with Hilbert--Schmidt embeddings of Sobolev spaces. Both non-homogeneous and homogeneous kernels are considered. In the latter case, results in a general Schatten class norm are also provided. Important examples of homogeneous kernels covered by the results of the paper include the class of Mat\'ern kernels. 
	\end{abstract}
	
	\maketitle
	\section{Introduction}
	\label{sec:intro}
	
	A Gaussian random field on a bounded domain~$\cD$ is characterized by its mean and its covariance. Depending on the research community, the covariance is described by a covariance kernel~$q$ or a covariance operator~$Q$. More specifically, given a symmetric continuous covariance kernel $q \colon \bar{\cD} \times \bar{\cD} \to \R$, the corresponding covariance operator~$Q$ is positive semidefinite and self-adjoint on the Hilbert space~$L^2(\cD)$ and given by
	\begin{equation*}
	Q u (x) = \int_{\cD} q(x,y) u(y) \dd y, 
	\end{equation*}
	for $x \in \cD$, $u \in L^2(\cD)$. Our main goal in this paper is to, given the regularity of the kernel~$q$, derive regularity estimates for~$Q$ in terms of certain smoothness spaces related to elliptic operators with boundary conditions on $\cD$.
	
	Our motivation to analyze the coupling of these two formulations in detail comes from the theory and approximation of solutions to stochastic partial differential equations (SPDEs). While the analysis of these equations and their numerical approximations is mainly done in Hilbert spaces, e.g., certain fractional order spaces related to the differential operator in the equation, with a $Q$-Wiener process as driving noise, algorithms that generate this driving noise in practice are often based on the covariance kernel~$q$. The class of Mat\'ern kernels is a popular example in spatial statistics. Surprisingly, to the best of our knowledge, such results are not available in the literature.
	
	To be able to put our abstract results and their consequences in a more specific context, let us consider a linear stochastic reaction-diffusion equation with additive noise
	\begin{equation}
	\label{eq:intro-spde}
	\begin{dcases*}
	\frac{\partial X}{\partial t}(t,x) = \sum_{i,j = 1}^d \frac{\partial}{\partial x_j} \left(a_{i,j}(\cdot) \frac{\partial X}{\partial x_i}\right)(t,x) - c(x) X(t,x) + \frac{\partial W}{\partial t}(t,x), (t,x) \in (0,T] \times \cD, \\
	X(0,x) = X_0(x), x \in \cD,
	\end{dcases*}
	\end{equation}
	on a convex bounded domain $\cD \subset \R^d, d= 1,2,3$, with boundary $\partial \cD$. Here the functions $(a_{i,j})_{i,j=1}^d, c$ fulfill an ellipticity condition, $X_0$ is some smooth initial function and homogeneous boundary conditions of either Dirichlet or Neumann type are considered. The stochastic noise term $\partial W / \partial t$ is Gaussian, white in time and correlated by a symmetric continuous covariance kernel $q \colon \bar{\cD} \times \bar{\cD} \to \R$ in space. This can be seen as a simplified version of equations considered for the modeling of sea surface temperature and other geophysical spatio-temporal processes on some spatial domain $\cD$ \cite[Chapter~6]{LR17}. This equation is considered in the context of~\cite{DPZ14} as a stochastic differential equation of It\^o type on the Hilbert space $H = L^2(\cD)$ of square integrable functions on $\cD$. The stochastic partial differential equation~\eqref{eq:intro-spde} is then written in the form
	\begin{equation}
	\label{eq:spde-intro-2}
	\dd X(t) + \Lambda X(t) = \dd W(t),
	\end{equation} 
	for $t \in (0,T]$. The unbounded linear operator $\Lambda$ on $H$ is densely defined, self-adjoint and positive definite with a compact inverse, see Section~\ref{sec:frac} for precise assumptions. The stochastic term $W$ is an $H$-valued $Q$-Wiener process on a complete filtered probability space $(\Omega, \cF, \IP)$. Here $Q$ is a positive semidefinite self-adjoint integral operator on $H$ with kernel $q$. If $x \mapsto W(1,x) = W(1)(x)$ is pointwise defined and jointly measurable with respect to the product $\sigma$-algebra $\cF \otimes \cB(\cD)$ (with $\cB(\cD)$ denoting the Borel $\sigma$-algebra on $\cD$) then $q$ is the covariance function of the random field $(W(1,x))_{x \in \cD}$. In general, there is no analytic solution to~\eqref{eq:spde-intro-2} so numerical approximations have to be computed. It is then vital to understand how various regularity properties of $q$ influence the behavior of $X$ and its approximation, since this can determine the convergence rate of the numerical approximations. We discuss this in concrete terms in Section~\ref{sec:spde}.
	
	The research field on SPDEs of the form~\eqref{eq:spde-intro-2} has been very active in the 21st century. There is a substantial body of literature, both from theoretical \cite{DPZ14,LR17} as well as numerical \cite{JK09,K14,LPS14} perspectives. For SPDEs on domains without boundary (i.e., when $\cD$ is replaced by Euclidean space, a torus or a sphere) the question of how regularity properties of $q$ influence $X$ is well understood, especially in the homogeneous case \cite{DF98, KZ00, LS15,PZ97}. This refers to the case that $q(x,y)$ only depends on the difference $x-y$ between two points $x$ and $y$ in $\cD$, examples including the class of Mat\'ern kernels, see Remark~\ref{rem:matern}. For domains with boundaries, such results are rarely found in the literature. Usually, the analysis is restricted to the special case that the eigenvalues and eigenfunctions of $Q$ are explicitly known. In particular it is common to consider the case that $Q$ and $\Lambda$ commute, see, e.g., \cite[Section~5.5.1]{DPZ14}. One of the few instances in which an author instead considers the properties of $q$ as a function on $\bar{\cD} \times \bar{\cD}$ when deriving connections between properties of $q$ and properties of $X$ can be found in~\cite{B05}. The main result of this paper is \cite[Theorem~4.2]{B05}, which states that $q$ has to satisfy the boundary conditions of $\Lambda$ in a certain sense in order for $Q$ and $\Lambda$ to commute. In practice, this excludes the physically relevant case of homogeneous noise from approaches such as that of \cite[Section~5.5.1]{DPZ14}, see \cite[Corollary~4.9]{B05}.  Our approach to the problem consists instead of deriving sufficient conditions on $q$ for which the associated symmetric operator $Q$ fulfills estimates of the form 
	\begin{equation}
	\label{eq:tracecond}
	\norm{\Lambda^{\frac{r}{2}} Q \Lambda^{\frac{r}{2}}}{\cL_1(H)} = \norm{\Lambda^{\frac{r}{2}} Q^{\frac{1}{2}}}{\cL_2(H)}^2 < \infty
	\end{equation}
	and
	\begin{equation}
	\label{eq:Qreg}
	\norm{\Lambda^{\frac{r}{2}} Q \Lambda^{\frac{s}{2}}}{\cL_2(H)} < \infty
	\end{equation}
	for fractional powers $\Lambda^{r/2}$ of $\Lambda$ and suitable constants $r,s \ge 0$. By $\cL_1(H)$ and $\cL_2(H)$ we denote the spaces of trace-class and Hilbert--Schmidt operators, respectively. We consider both homogeneous and non-homogeneous kernels $q$. In the former case, we are able to deduce estimates of the form~\eqref{eq:Qreg} with the $\cL_2(H)$-norm replaced by the more general Schatten class $\cL_p(H)$-norm, $p \ge 1$. 
	
	In the setting of~\eqref{eq:spde-intro-2}, the condition on $Q$ in~\eqref{eq:tracecond} is for a given value of $r \ge 0$ equivalent to requiring that the $H$-valued random variable $W(t)$ takes values in the subspace $\dot{H}^r = \dom(\Lambda^{r/2})$ of the fractional Sobolev space $H^r = W^{r,2}(\cD)$ at all times $t \in [0,T]$. This is a commonly encountered assumption in the literature, particularly when analyzing numerical approximation schemes for SPDEs, see Section~\ref{sec:spde}. The condition also has implications for the qualitative behavior of $X$. It guarantees that $X$ takes values in $\dot{H}^{r+1} \subset H^{r+1}$ \cite[Proposition~6.18]{DPZ14} with sample paths continuous in $\dot{H}^{r+1-\epsilon} \subset H^{r+1-\epsilon}$ for arbitrary $\epsilon>0$ \cite[Theorem~5.15]{DPZ14}. In particular, if~\eqref{eq:tracecond} holds with $r > 1/2$, $X$ is a strong solution (in the PDE sense) to~\eqref{eq:spde-intro-2} as opposed to just a weak solution, cf.\ \cite[Theorem~5.40]{DPZ14}. This means that the process $X$ takes values in $\dom(\Lambda)$ and has the intuitive representation
	\begin{equation*}
	X(t) = W(s) - \int_0^t \Lambda X(s) \dd s
	\end{equation*}
	for $t \in (0,T]$. Moreover, if~\eqref{eq:tracecond} holds with $r > d/2-1$, a classical Sobolev embedding theorem (see, e.g., \cite[Section~8]{NPV12}) ensures that $X(t)$ takes values in the space $\cC^{0,\max(1,\sigma)}(\bar{\cD})$ of H\"older continuous functions on $\bar{\cD} = \cD \cup \partial \cD$ with exponent $\sigma \in (0,r+1-d/2]$ $\IP$-a.s.. The evaluation functional is continuous on this Banach space. It follows that for each $t>0$, $(X(t,x))_{x \in \bar{\cD}}$ is a smooth random field on $\bar{\cD}$ as opposed just an abstract random variable in a Hilbert space, c.f.\ \cite[Sections~7.4-7.5]{HE15}. In summary, there could be many reasons why one would like to know for which $r>0$ the estimate~\eqref{eq:tracecond} is satisfied for a given kernel $q$. In the case of non-homogeneous kernels, we consider H\"older conditions on the kernel $q$ when deriving the estimate~\eqref{eq:tracecond}. This is a natural choice given our proof technique, which is based on the fact that $Q^{1/2}(H)$, the image of the square root $Q^{1/2}$ of $Q$, coincides with the reproducing kernel Hilbert space of functions on $\bar{\cD}$ associated with the kernel $q$ \cite{W04}. For a similar reason, in the case of homogeneous kernels $q$, we consider a decay condition on the Fourier transform of $q$, which is used in practice when generating the driving noise with fast Fourier transforms (cf.~\cite{LP11}). While this interpretation of $Q^{1/2}(H)$ as a reproducing kernel Hilbert space is well-known, we are not aware that it has been used to find estimates of the form~\eqref{eq:tracecond} before. 
	
	The estimate~\eqref{eq:Qreg} does not, as far as we know, have an immediate interpretation in terms of regularity properties of $X$ or $W$. It is, however, important for analyzing weak errors of approximations to certain SPDEs, such as the stochastic wave equation. It has also been used in the recent work~\cite{KLP21a} to derive higher convergence rates for approximations of the covariance operator of SPDE solutions. There is an immediate connection between the condition on $Q$ in~\eqref{eq:Qreg} and regularity of $q$: the condition is true with $r=s$ if and only if $q$ is an element of the Hilbert  tensor product space $\dot{H}^r \otimes \dot{H}^r$, see~\cite{CK20,T67}. Instead of exploiting this connection, we consider H\"older or Fourier transform conditions on $q$ also in this case. The reason for this is partly that these conditions are easier to check in applications compared to the rather abstract tensor product condition. We also want to ensure easy comparisons between the estimates~\eqref{eq:tracecond} and~\eqref{eq:Qreg} under the same conditions on $q$.  
	
	The outline of the paper is as follows. The next section contains an introduction to the necessary mathematical background along with our assumptions on $\Lambda$ and $Q$. This includes short introductions to fractional powers of elliptic operators on bounded domains and reproducing kernel Hilbert spaces along with the proofs of some preliminary lemmas. We derive the estimates~\eqref{eq:tracecond} and~\eqref{eq:Qreg} under H\"older conditions on a non-homogeneous kernel $q$ in Section~\ref{sec:non-hom}. In Section~\ref{sec:hom} we consider a decay condition on the Fourier transform of a homogeneous kernel $q$ and derive estimates on \begin{equation*}
	\norm{\Lambda^{\frac{r}{2}} Q \Lambda^{\frac{s}{2}}}{\cL_p(H)} < \infty
	\end{equation*}
	for $p \ge 1$, a more general form of the estimate~\eqref{eq:Qreg}. We use this estimate to obtain conditions for which the estimate~\eqref{eq:tracecond} is satisfied. Section~\ref{sec:spde} concludes the paper with a discussion of the implication of our results for the numerical analysis of SPDEs on domains with boundary. The applications we discuss are not limited to stochastic reaction-diffusion equations but include several other SPDEs involving elliptic operators on bounded domains, such as stochastic Volterra equations or stochastic wave equations.

	Throughout the paper, we adopt the notion of generic constants, i.e., the symbol $C$ is used to denote a positive and finite number which may vary from occurrence to occurrence and is independent of any parameter of interest. We use the expression $a \lesssim b$ to denote the existence of a generic constant $C$ such that $a \le C b$.
	
	\section{Preliminaries}
	
	In this section, we introduce our notation and reiterate some important results that we use in Sections~\ref{sec:non-hom}-\ref{sec:hom}. The material mainly comes from \cite[Chapter~1]{BT04}, \cite[Section~1.3-1.4]{G85}, \cite[Appendix~B]{K14} and~\cite[Chapters~1-2]{Y10}. We give explicit references for vital or nonstandard results.
	
	\subsection{Schatten class operators}
	
	Let $H$ and $U$ be real separable Hilbert spaces. By $\cL(H,U)$ we denote the space of linear and bounded operators from $H$ to $U$ and by $\cL_p(H,U)$ the subspace of Schatten class operators of order $p \in [1,\infty)$. This is a separable Banach space of compact operators with norm characterized by 
	\begin{equation}
	\label{eq:singval-characterization}
	\norm{\Gamma}{\cL_{p}(H,U)} = \left(\sum_{j=1}^\infty (\lambda_j(\Gamma))^p\right)^{\frac{1}{p}}.
	\end{equation} 
	Here $(\lambda_j(\Gamma))_{j = 1}^\infty$ are the singular values of $\Gamma \in \cL_{p}(H,U)$, i.e., the square roots of the eigenvalues of $\Gamma^*\Gamma$, which form a non-increasing sequence with limit $0$. The special cases $\cL_1(H,U)$ and $\cL_2(H,U)$ are referred to as the trace-class and Hilbert--Schmidt operators, respectively. 
	
	For an additional real separable Hilbert space $V$, let $\Gamma_1 \in \cL(U,V)$ and $\Gamma_2 \in \cL(V,H)$ be compact operators. Then, for any $j,k \ge 0$, 
	\begin{equation*}
	\lambda_{j + k + 1}(\Gamma_1 \Gamma_2) \le \lambda_{j+1}(\Gamma_1) \lambda_{k+1}(\Gamma_2).
	\end{equation*}
	This inequality is proven in \cite[Theorem~2]{F51} for the case that $U = V = H$, but the proof is readily adapted to our situation using the fact that the eigenvalues of $\Gamma^* \Gamma$ and $\Gamma \Gamma^*$, for a given operator $\Gamma$ between Hilbert spaces, coincide \cite[Section~4.3]{HE15}. Using the H\"older inequality for sequence spaces $\ell^p$, $p \ge 1$, it follows that for $p, q, r \in [1,\infty)$ with $1/r = 1/p + 1/q$, if $\Gamma_1 \in \cL_q(V,H)$ and $\Gamma_2 \in \cL_p(U,V)$, then $\Gamma_1 \Gamma_2 \in \cL_r(U,H)$ and 
	\begin{equation}
	\label{eq:schatten-holder}
	\norm{\Gamma_1 \Gamma_2}{\cL_r(U,H)} \le 2^{1/r} \norm{\Gamma_1}{\cL_q(V,H)} \norm{\Gamma_2}{\cL_p(U,V)} .
	\end{equation}
	Moreover, if $E$ is an additional real separable Hilbert space, $\Gamma_1 \in \cL(H,E)$, $\Gamma_3 \in \cL(U,V)$ (not necessarily compact) and $\Gamma_2 \in \cL_p(V,H)$ for some $p \in [1,\infty)$, then $\Gamma_1 \Gamma_2 \Gamma_3 \in \cL_p(U,E)$ with 
	\begin{equation}
	\label{eq:schatten-ideal}
	\norm{\Gamma_1 \Gamma_2 \Gamma_3}{\cL_p(U,E)} \le  \norm{\Gamma_1}{\cL(H,E)} \norm{\Gamma_2}{\cL_p(V,H)} \norm{\Gamma_3}{\cL(U,V)}.
	\end{equation}
	This ideal property follows from the fact that
	\begin{equation*}
	\lambda_{j}(\Gamma_1 \Gamma_2 \Gamma_3) \le \norm{\Gamma_1}{\cL(H,E)} \lambda_{j}(\Gamma_2 \Gamma_3) \le \norm{\Gamma_1}{\cL(H,E)} \lambda_{j}(\Gamma_2) \norm{\Gamma_3}{\cL(U,V)},
	\end{equation*}
	where the inequalities are consequences of the min-max theorem.
	
	The space $\cL_2(H,U)$ is a separable Hilbert space with inner product
	\begin{equation*}
	\inpro[\cL_2(H,U)]{\Gamma_1}{\Gamma_2} = \sum_{j = 1}^{\infty} \inpro[U]{\Gamma_1 e_j}{\Gamma_2 e_j} = \trace(\Gamma_2^* \Gamma_1) = \trace(\Gamma_1^* \Gamma_2),
	\end{equation*}
	for $\Gamma_1, \Gamma_2 \in \cL_2(H,U)$ and an arbitrary orthonormal basis $(e_j)_{j = 1}^\infty$ of $H$. For a positive semidefinite symmetric operator $Q \in \cL(H)$, $\norm{Q}{\cL_1(H)} = \trace(Q)$.  
	
	\subsection{Fractional powers of elliptic operators on bounded domains}
	\label{sec:frac}
	
	Consider a Gelfand triple $V \subset H \subset V^*$ of real separable Hilbert spaces with dense and continuous embeddings. We assume that $(V^*,V)$ forms an adjoint pair with duality product ${}_V \langle \cdot, \cdot \rangle_{V^*}$ such that ${}_{V^*} \langle u, v \rangle_V = \inpro[H]{u}{v}$ for all $u \in H$, $v \in V$. Given $V$ and $H$, $V^*$ can be constructed by identifying $H$ with its continuous dual space $H' \subset V'$ and letting $V^*$ be the completion of $H$ under the norm of $V'$. Next, let $\lambda \colon V \times V \to \R$ be a continuous, symmetric and coercive bilinear form, i.e., it is linear in both arguments, $\lambda(u,v) = \lambda(v,u)$ and there are constants $C, c > 0$ such that $|\lambda(u,v)| \le C \norm{u}{V} \norm{v}{V}$ and $\lambda(u,u) \ge c \norm{u}{V}^2$ for all $u,v \in V$. Then, there exists a unique isomorphism $L \colon V \to V^*$ such that ${}_{V^*} \langle L u, v \rangle_V = \lambda(u,v)$ for all $u, v \in V$. Viewing $L$ as an operator on $V^*$, it is densely defined and closed. We write $\Lambda = L|_{H}$ for its part in $H$, which is then a densely defined, closed, self-adjoint and positive definite operator with domain $D(\Lambda) = \{v \in V : \Lambda v = L v \in H\}$. It has a self-adjoint inverse $\Lambda^{-1} \in \cL(H)$, which we assume to be compact. 
	
	Applying the spectral theorem to $\Lambda^{-1}$, we obtain the existence of a sequence $(\lambda_j)_{j=1}^\infty$ of positive non-decreasing eigenvalues of $\Lambda$, along with an accompanying orthonormal basis of eigenfunctions $(e_j)_{j=1}^\infty$ in $H$. For $r \ge 0$, we define fractional powers of $\Lambda$ by 
	\begin{equation}
	\label{eq:fracpower}
	\Lambda^{\frac{r}{2}} v = \sum_{j = 0}^{\infty} \lambda_j^{\frac{r}{2}} \inpro[H]{v}{e_j} e_j
	\end{equation}
	for 
	\begin{equation*}
	v \in \dom(\Lambda^{\frac{r}{2}}) = \left\{ v \in H : \norm{x}{\dot{H}^r}^2 = \sum^\infty_{j=1} \lambda_j^r \inpro[H]{v}{e_j}^2 < \infty\right\}.
	\end{equation*}
	We write $\dot{H}^r = \dom(\Lambda^{{r/2}})$. Note that $\dot{H}^0 = H$. For $r<0$, we let $\Lambda^{{r/2}}$ be defined by~\eqref{eq:fracpower},  and we write $\dot{H}^r$ for the completion of $H$ under the norm $\norm{x}{\dot{H}^r}$. Equivalently, we can write
	\begin{equation*}
	\dot{H}^r = \dom(\Lambda^{\frac{r}{2}}) = \left\{x = \sum_{j=1}^{\infty} x_j e_j : (x_j)_{j=1}^\infty \subset \R \text{ such that } \norm{x}{\dot{H}^r}^2 = \sum_{j=1}^{\infty} \lambda_j^{r} x_j^2 < \infty \right\}.
	\end{equation*} 
	Then $\Lambda^{{r/2}} \in \cL(\dot{H}^r,H)$ for all $r \in \R$. Regardless of the sign of $r$, $\dot{H}^r$ is a Hilbert space with inner product $\inpro[\dot{H}^r]{u}{v} = \inpro[H]{\Lambda^{r/2}u}{\Lambda^{r/2}v}$. For $r>0$, $\dot{H}^{-r}$ is isometrically isomorphic to $( \dot{H}^r)'$ \cite[Theorem~B.8]{K14}. In particular, for $v \in H = H'$, 
	\begin{equation*}
	\norm{v}{\dot{H}^{-r}} = \norm{\Lambda^{-\frac{r}{2}} v}{H} = \sup_{u \in \dot{H}^r} \frac{|\inpro[H]{u}{v}|}{\norm{u}{\dot{H}^{r}}}.
	\end{equation*}
	For $r > s$, we have $\dot{H}^r \hookrightarrow \dot{H}^s$ where the embedding is dense and continuous. Since $(\Lambda^{-{r/2}} e_j)_{j=1}^\infty$ and $(\Lambda^{-{s/2}} e_j)_{j=1}^\infty$ are orthonormal bases of $\dot{H}^r$ and $\dot{H}^s$, respectively, we may represent the embedding $I_{\dot{H}^r \hookrightarrow \dot{H}^s}$ by
	\begin{equation*}
	I_{\dot{H}^r \hookrightarrow \dot{H}^s} = \sum_{j=1}^{\infty} \inpro[\dot{H}^r]{\cdot}{\Lambda^{-\frac{r}{2}} e_j} \Lambda^{-\frac{r}{2}} e_j = \sum_{j=1}^{\infty} \lambda_j^{\frac{s-r}{2}} \inpro[\dot{H}^r]{\cdot}{\Lambda^{-\frac{r}{2}} e_j} \Lambda^{-\frac{s}{2}} e_j.
	\end{equation*}
	It follows that $I_{\dot{H}^r \hookrightarrow \dot{H}^s}$ is compact. Lemma~2.1 in~\cite{BKK20} allows us to, for all $r, s \in \R$, extend $\Lambda^{{s/2}}$ to an operator in $\cL(\dot{H}^r,\dot{H}^{r-s})$, and we will do so without changing notation. 
	
	Let us write $\dot{H}^s_{\C}$ for the complexification of $\dot{H}^s$, i.e., the complex Hilbert space given by $\dot{H}^s \times \dot{H}^s$ with scalar multiplication $(a+bi)(u,v) = (au-bv,av+bu)$ for $a+bi \in \C$ and $(u,v) \in \dot{H}^s_{\C}$. We write $u + iv$ for $(u,v)$ so that we may consider $(\Lambda^{-s/2} e_j)_{j=1}^\infty$ as an orthonormal basis of $\dot{H}^s_{\C}$ when this is equipped with the inner product 
	\begin{equation*}
	\inpro[\dot{H}^s_{\C}]{u + iv}{x + iy} = \inpro[\dot{H}^s]{u}{x} + i \inpro[\dot{H}^s]{v}{x} - i \inpro[\dot{H}^s]{u}{y}+ \inpro[\dot{H}^s]{v}{y}.
	\end{equation*}
	For $r\ge s \in \R$, it follows that
	\begin{equation*}
	\dot{H}^r_{\C} = \left\{ v \in \dot{H}^s_{\C} : \sum^\infty_{j=1} \lambda_j^{r-s} |\inpro[\dot{H}^s_{\C}]{v}{\Lambda^{-s/2} e_j}|^2 < \infty\right\}.
	\end{equation*}
	From this, one obtains (see, e.g., \cite{CG94}) that for $\theta \in [0,1]$, $\dot{H}_\C^{\theta r + (1-\theta) s} = [\dot{H}_\C^{s},\dot{H}_\C^{r}]_\theta$ in the sense of isometric isomorphisms, where $[\dot{H}_\C^{s},\dot{H}_\C^{r}]_\theta$ is defined by complex interpolation. We refer to $u$ and $v$ as the real and imaginary parts of $u + iv \in \dot{H}^s_{\C}$ and write $\overline{u + iv}$ for $u - iv$. The real part of $\dot{H}^s_{\C}$ with only real multiplication of scalars is a Hilbert space isometrically isomorphic to $\dot{H}^s$.
	
	The interpolation space representation of $\dot{H}_\C^{\theta r + (1-\theta) s}$ allows us to prove the following lemma for a symmetric operator $Q \in \cL(H)$.
	\begin{lemma}
		\label{lem:symmetricschatten}
		Let $Q \in \cL(H)$ be symmetric. Then the following three claims are equivalent for $p \in [1,\infty)$ and $r \in \R$:
		\begin{enumerate}[label=(\roman*)]
			\item \label{lem:symmetricschatten:i} $\norm{Q}{\cL_p(H,\dot{H}^r)} = \norm{\Lambda^{{r/2}} Q}{\cL_p(H)} < \infty$.
			\item \label{lem:symmetricschatten:ii} $\norm{Q}{\cL_p(\dot{H}^{-r},H)} = \norm{Q \Lambda^{{r/2}}}{\cL_p(H)} < \infty$.
			\item \label{lem:symmetricschatten:iii} $\norm{Q}{\cL_p(\dot{H}^{(\theta-1) r},\dot{H}^{\theta r})} = \norm{\Lambda^{\theta r/2}Q \Lambda^{(1-\theta) r/2}}{\cL_p(H)} < \infty$ for all $\theta \in [0,1]$.
		\end{enumerate}
	\end{lemma}
	\begin{proof}
		Assume first that $r\ge 0$. If $Q \in \cL(H,\dot{H}^r)$, then by symmetry of $Q$, 
		\begin{equation*}
		\norm{Q v }{H} = \sup_{\substack{w \in H \\ \norm{w}{H}=1}}|\inpro[H]{Q v}{w}| = \sup_{\substack{w \in H \\ \norm{w}{H}=1}}|\inpro[H]{ \Lambda^{-\frac{r}{2}}v}{\Lambda^{\frac{r}{2}}Qw}| \le  \norm{v}{\dot{H}^{-r}} \norm{Q}{\cL(H,\dot{H}^r)}
		\end{equation*}
		so that $Q$ can be continuously extended to $\dot{H}^{-r}$. Conversely, if $Q$ extends to $\dot{H}^{-r}$, the operator norm $\norm{Q\Lambda^{r/2}}{\cL(H)}$ is finite and therefore 
		\begin{align*}
		\norm{Qv}{\dot{H}^r}^2 = \sum_{j = 1}^\infty \lambda_j^r \inpro[H]{Q v}{e_j}^2 = \sum_{j = 1}^\infty \inpro[H]{v}{Q \Lambda^{\frac{r}{2}} e_j}^2  = \sum_{j = 1}^\infty \inpro[H]{(Q \Lambda^{\frac{r}{2}})^*v}{e_j}^2 \le \norm{(Q \Lambda^{\frac{r}{2}})^*}{\cL(H)} \norm{v}{H}
		\end{align*}
		so that $Q \in \cL(H,\dot{H}^r)$.
		Let us now write $\hat Q^*$ 
		for the adjoint of $Q$ with respect to $\cL(H,\dot{H}^r)$, i.e., the operator in $\cL(\dot{H}^r,H)$ defined by
		\begin{equation*}
		\inpro[H]{\hat Q^* u}{v} = \inpro[\dot{H}^r]{u}{Q v} = \inpro[H]{\Lambda^{\frac{r}{2}}u}{\Lambda^{\frac{r}{2}}Q v}
		\end{equation*}
		for $u \in \dot{H}^r, v \in H$. Let us similarly write $\check Q^* \in \cL(H,\dot{H}^{-r})$ for the adjoint of $Q$ with respect to $\cL(\dot{H}^{-r},H)$. We have for $u \in H$, $v \in \dot H^r$ that 
		\begin{equation}
		\label{eq:lem:symmetricschatten:eq1}
		\inpro[H]{\Lambda^{-\frac{r}{2}} \check Q^* u}{v} = \inpro[\dot{H}^{-r}]{\check Q^* u}{\Lambda^{\frac{r}{2}}v} = \inpro[H]{u}{Q\Lambda^{\frac{r}{2}}v} = \inpro[H]{\Lambda^{\frac{r}{2}} Q u}{v}
		\end{equation}
		so that by density of $\dot{H}^r$ in $H$, $\Lambda^{-{r/2}} \check Q^* = \Lambda^{{r/2}} Q$. This implies  
		\begin{align*}
		\norm{(\hat Q^* Q - Q \check Q^* ) u}{H} &= \sup_{\substack{v \in H \\ \norm{v}{H}=1}} | \inpro[H]{\hat Q^* Q u }{v} - \inpro[H]{Q \check Q^* u }{v}| \\
		&= \sup_{\substack{v \in H \\ \norm{v}{H}=1}} | \inpro[H]{\Lambda^\frac{r}{2} Q u }{\Lambda^\frac{r}{2} Q v} - \inpro[H]{\Lambda^{-\frac{r}{2}} \check Q^* u }{\Lambda^\frac{r}{2} Q v}| = 0 
		\end{align*}
		for $u \in H$, hence $Q \check Q^* = \hat Q^* Q$. Since the eigenvalues of $Q \check Q^*$ are the same as those of $\check Q^* Q$ \cite[Section~4.3]{HE15}, the singular values of $Q \in \cL(H,\dot{H}^r)$ and $Q \in \cL(\dot{H}^{-r},H)$ are the same, so that~\eqref{eq:singval-characterization} implies the equivalence of \ref{lem:symmetricschatten:i} and \ref{lem:symmetricschatten:ii} for $p \in [1,\infty)$. For $r<0$ the proof is the same, except that we take $v\in H$ in~\eqref{eq:lem:symmetricschatten:eq1}. 
		
		Clearly \ref{lem:symmetricschatten:iii} implies \ref{lem:symmetricschatten:ii} and \ref{lem:symmetricschatten:i}. For the other direction, we first write $Q_\C$ for the compact operator in $\cL(H_\C,\dot{H}^r_\C)$ (and $\cL({\dot{H}^{-r}}_\C,H_\C)$) defined by $Q_\C (u + iv) = Qu + i Qv$. Then we note that $\dot{H}_\C^{\theta r} = [H_\C,\dot{H}_\C^r]_\theta$ and $\dot{H}_\C^{(\theta-1) r} = [\dot{H}_\C^{-r},H_\C]_\theta$ in the case that $r > 0$ while $\dot{H}_\C^{\theta r} = [\dot{H}_\C^r,H_\C]_{1-\theta}$ and $\dot{H}_\C^{(\theta-1) r} = [H_\C,\dot{H}_\C^{-r}]_{1-\theta}$ in the case that $r < 0$. Therefore, $Q_\C \in \cL(\dot{H}_\C^{(\theta-1) r},\dot{H}_\C^{\theta r})$ and by restriction to the real elements of $\dot{H}_\C^{(\theta-1) r}$, it follows that $Q \in \cL(\dot{H}^{(\theta-1) r},\dot{H}^{\theta r})$. For a fix $\theta \in [0,1]$, let us regard $Q_\C$ and $Q$ as operators in these spaces. Note that $Q_\C$ is compact if and only if $Q$ is compact. Moreover, the vector $v \in \dot{H}_\C^{(\theta-1) r}$ is an eigenvector of $Q_\C^* Q_\C$ with (real and positive) eigenvalue $\mu$ if and only if $\inpro[\dot{H}_\C^{\theta r}]{Q_\C v}{Q_\C u} = \mu \inpro[\dot{H}_\C^{(\theta-1) r}]{v}{u}$ for all $u \in \dot{H}_\C^{(\theta-1) r}$. From this it follows that $\bar{v}$ is also an eigenvector of $Q_\C^* Q_\C$ and that $(v + \bar{v})/2$ is an eigenvector of $Q^* Q$. Similarly, if $u \in \dot{H}^{(\theta-1) r}$ is an eigenvector of $Q^* Q$, it is also an eigenvector of $Q_\C^* Q_\C$ in $\dot{H}_\C^{(\theta-1) r}$ with the same eigenvalue. In other words, the sets of singular values of $Q_\C \in \cL(\dot{H}_\C^{(\theta-1) r},\dot{H}_\C^{\theta r})$ and $Q \in \cL(\dot{H}^{(\theta-1) r},\dot{H}^{\theta r})$ coincide. Therefore, $Q_\C \in \cL_p(\dot{H}_\C^{-r},H_\C) \cap \cL_p(H_\C,\dot{H}_\C^{r})$, which, using the results of~\cite{G74}, implies that $Q_\C \in \cL_p(\dot{H}_\C^{(\theta-1) r},\dot{H}_\C^{\theta r})$ for all $\theta \in [0,1]$. Another appeal to the correspondence of the singular values of $Q$ and $Q_\C$ finishes the proof.
	\end{proof}
	
	In the spirit of the previous result, let us also note that the equality in~\eqref{eq:tracecond} is true when $Q$ is also assumed to be positive semidefinite (so that $Q^{1/2}$ is well-defined). Indeed, the right hand side guarantees that $Q^{1/2}$ extends to $\dot H^{-r}$ so for $u \in H$
	\begin{equation*}
	\norm{\Lambda^{\frac{r}{2}} Q u}{H} \le \norm{\Lambda^{\frac{r}{2}} Q^{\frac{1}{2}}}{\cL(H)} \norm{Q^{\frac{1}{2}}\Lambda^{\frac{r}{2}}}{\cL(H)} \norm{\Lambda^{-\frac{r}{2}} u}{H}.
	\end{equation*}
	Therefore $\Lambda^{{r/2}} Q$ also extends to $\dot H^{-r}$. The operator $\Lambda^{{r/2}} Q \Lambda^{{r/2}}$ is symmetric and positive semidefinite on $H$, so (since $(\Lambda^{{r/2}} Q^{{1/2}})^* e_j = Q^{{1/2}} \Lambda^{{r/2}} e_j$ on an eigenfunction $e_j$ of $\Lambda$) we have 
	\begin{align*}
	\norm{\Lambda^{\frac{r}{2}} Q \Lambda^{\frac{r}{2}}}{\cL_1(H)} &= \sum_{j = 1}^\infty \inpro[H]{\Lambda^{\frac{r}{2}} Q \Lambda^{\frac{r}{2}} e_j}{e_j} \\
	&= \sum_{j = 1}^\infty \norm{Q^{\frac{1}{2}} \Lambda^{\frac{r}{2}} e_j}{H}^2 = \sum_{j = 1}^\infty \norm{(\Lambda^{\frac{r}{2}} Q^{\frac{1}{2}})^* e_j}{H}^2 = \norm{(\Lambda^{\frac{r}{2}} Q^{\frac{1}{2}})^*}{\cL_2(H)}^2 = \norm{\Lambda^{\frac{r}{2}}Q^{\frac{1}{2}}}{\cL_2(H)}^2.
	\end{align*}
	
	Before we put the abstract framework above into a concrete setting, we note some properties of fractional Sobolev spaces. We denote by $H^m(\cD) = W^{m,2}(\cD)$ and $H^m(\R^d) = W^{m,2}(\R^d)$ the classical Sobolev spaces of order $m$, on a bounded domain $\cD\subset\R^d$ with Lipschitz boundary and $\R^d$, $d \in \N$, respectively. We consider only Sobolev spaces of real-valued functions in this paper. The results we cite are generally proved for Sobolev spaces of complex-valued functions but naturally hold also in this real-valued setting. The norm of $H^m(\R^d)$ is given by
	\begin{equation*}
	\norm{u}{H^m(\R^d)}^2 = \sum_{|\alpha|\le m} \norm{D^\alpha u}{L^2(\R^d)}^2,
	\end{equation*}
	where $D^\alpha$ is the weak derivative with respect to a multiindex $\alpha=(\alpha_1,\dots,\alpha_d)$
	and the norm of $H^m(\cD)$ is defined in the same way. When there is no risk of confusion we write $H^m$ for $H^m(\cD)$ and we set $H^0 = H = L^2(\cD)$. For $s = m + \sigma$, $m \in \N_0$, $\sigma \in (0,1)$, we use the same notation for the fractional Sobolev space $H^s$ as for the classical Sobolev space. The space $H^s$ is equipped with the Sobolev--Slobodeckij norm
	\begin{equation*}
	\norm{u}{H^s} = \left(\norm{u}{H^m}^2 + \sum_{|\alpha| = m} \int_{\cD \times \cD} \frac{|D^\alpha u(x) - D^\alpha u(y)|^2}{|x-y|^{d+2\sigma}}\dd x \dd y\right)^{1/2},
	\end{equation*}
	for $u \in H^s$. Since $\partial \cD$ is Lipschitz, an equivalent \cite[page~25]{G85} norm is given by
	\begin{equation}
	\label{eq:fouriernormD}
	\norm{u}{H^s(\cD)} = \inf_{\substack{v \in H^s(\R^d) \\ v|_{\cD}=u}} \norm{v}{H^s(\R^d)}.
	\end{equation} 
	Here the norm of the fractional Sobolev space $H^s(\R^d)$ is given by 
	\begin{equation}
	\label{eq:fouriernorm}
	\norm{u}{H^s(\R^d)}^2 = \frac{1}{(2 \pi)^{\frac{d}{2}}} \int_{\R^d} |\hat{u}(\xi)|^2 (1 + |\xi|^2)^s \dd \xi,
	\end{equation} 
	where $\hat{u}$ is the Fourier transform of $u \in L^2(\R^d)$. 
	
	The embedding of $H^r$ into $H^s$ is dense and, since $\partial \cD$ is Lipschitz, compact for all $r > s \ge 0$ \cite[Theorem~1.4.3.2]{G85}. This last fact gives us yet another way to characterize $H^r$ and extend the definition to negative $r$. Since, for $m \in \N$, $H^m$ is densely and continuously embedded in $H$, there is a positive definite self-adjoint operator $\Theta_m$ such that $\dom(\Theta_m) = H^m$ and $\norm{\Theta^m u}{H} = \norm{u}{H^m}$ \cite[Section~1.2]{LM72}. The compactness of the embedding implies \cite[Sections~4.5.2-4.5.3]{T92} that $\Theta_m^{-1}$ is compact and the spaces $H_m^{r}=\dom(\Theta_m^{{r/2}})$, $r \in \R^+$, can now be constructed by the spectral decomposition of the operator. The space $H_m^{-r}$ is, like $\dot{H}^{-r}$, defined as the completion of $H$ with respect to the norm $\norm{\Theta_m^{-{r/2}}}{H}$. Naturally, $H^r_{\C}$ is isometrically isomorphic to the corresponding fractional Sobolev space of complex-valued functions on $\cD$. By~\cite[Theorem~1.35]{Y10} we then obtain, for $r \in [0,2]$, $H_{m,\C}^{r} = [H_{\C},H_{\C}^{m}]_{r/2} = H^{rm/2}_{\C}$. Thus, taking the real parts of these spaces, we obtain $H_{m}^{r} = H^{rm/2}$ in the sense of an isometric isomorphism. We define $H^{-r}$ for $r \in [0,m]$ by $H^{-2r/m}_m$, and note that this definition is independent of $m$ due to the fact that $H^{-r} = (H_m^{2r/m})'=(H^{r})'$. By this characterization of the norm of $H^r$, $r \in \R$, we can repeat the proof of Lemma~\ref{lem:symmetricschatten} to obtain the following analogous result.
	\begin{lemma}
		\label{lem:symmetricschatten2}
		Let $Q \in \cL(H)$ be symmetric. Then the following three claims are equivalent for $p \in [1,\infty)$ and $r \in \R$:
		\begin{enumerate}[label=(\roman*)]
			\item %
			$\norm{Q}{\cL_p(H,{H}^r)}< \infty$.
			\item %
			$\norm{Q}{\cL_p({H}^{-r},H)}< \infty$.
			\item %
			$\norm{Q}{\cL_p({H}^{(\theta-1) r},{H}^{\theta r})} < \infty$ for all $\theta \in [0,1]$.
		\end{enumerate}
	\end{lemma}
	The same characterization also allows us to deduce the following lemma, which tells us when the embedding $H^r\hookrightarrow H^s$, $r>s$, is of a particular Schatten class.
	\begin{lemma}
		\label{lem:sobolevembedding}
		For $ -\infty < s < r < \infty$, the embedding $I_{H^r \hookrightarrow H^s}$ fulfills:
		\begin{equation*}
		\norm{I_{H^r \hookrightarrow H^s}}{\cL_p(H^r, H^s)} < \infty \iff r -s > d/p.
		\end{equation*}
	\end{lemma}
	\begin{proof}
		We pick some $m \in \N$ such that $m \ge \max(|r|,|s|)$ and write $(f_j)_{j=1}^\infty$ for the orthonormal basis of eigenfunctions of the operator $\Theta_m$ that is associated with the space $H^m$. Then
		\begin{equation*}
		v \in H^r \iff v \in H_m^{\frac{2r}{m}} \iff \sum_{j = 1}^{\infty} \left|\inpro[H]{\Theta_m^{\frac{r}{m}} v}{ f_j}\right|^2 < \infty \iff \sum_{j = 1}^{\infty} \left|\inpro[H^s]{\Theta_m^{\frac{r-s}{m}} v}{ \Theta_m^{-\frac{s}{m}} f_j}\right|^2  < \infty,
		\end{equation*}
		which gives $H^r = \dom(\Theta_m^{(r-s)/m})$, $\Theta_m^{(r-s)/m}$ being considered as an operator on $H^s$. Therefore, \cite[Lemma~3]{T67} implies that the singular values of $I_{H^r \hookrightarrow H^s}$ coincide with the eigenvalues of $\Theta^{(s-r)/m}$. By the same argument, the singular values of $I_{H^{r-s} \hookrightarrow H}$ coincide with the eigenvalues of the same operator. It therefore suffices to show the claim for $I_{H^{r-s} \hookrightarrow H}$.
		
		The result is proven for $\partial \cD \in \cC^\infty$ in~\cite{T67}, specifically as a consequence of~\cite[Lemma~2, Satz~2]{T67}. For the general case, we first note that when $\partial \cD$ is Lipschitz, there is an operator $P \in \cL(H^{r-s}(\cD),H^{r-s}(\R^d))$, such that $P u |_{\cD} = u$ \cite[Theorem~1.4.3.1]{G85}. For an arbitrary bounded domain~$\cG \supset \cD$ such that $\partial \cG \in \cC^\infty$, we have
		\begin{equation*}
		I_{H^{r-s}(\cD) \hookrightarrow L^2(\cD)} = I_{L^2(\cG) \hookrightarrow L^2(\cD)} I_{H^{r-s}(\cG) \hookrightarrow L^2(\cG)} I_{H^{r-s}(\R^d) \hookrightarrow H^{r-s}(\cG)} P.
		\end{equation*}
		By definition of the norm~\eqref{eq:fouriernorm}, we note that $I_{H^{r-s}(\R^d) \hookrightarrow H^{r-s}(\cG)}$ is bounded, and clearly the restriction $I_{L^2(\cG) \hookrightarrow L^2(\cD)}$ is, too. Therefore, sufficiency of $r-s>d/p$ follows as a consequence of~\eqref{eq:schatten-ideal}. Necessity follows by an analogous argument: if $I_{H^{r-s}(\cD) \hookrightarrow L^2(\cD)} \in \cL_{p}(H^{r-s}(\cD),L^2(\cD))$ then  $I_{H^{r-s}(\tilde \cG) \hookrightarrow L^2(\tilde \cG)} \in \cL_{p}(H^{r-s}(\tilde \cG),L^2(\tilde \cG))$ for $p \in [1,\infty)$ and any domain $\tilde \cG \subset \cD$ with $\partial \tilde \cG \in \cC^\infty$.
	\end{proof}
	
	We now let the spaces $\dot{H}^r$, $r \ge 0$, obtain a concrete meaning by taking $\lambda$ to be a bilinear form on $V \subset H^1(\cD) \subset L^2(\cD)=H$ given by
	\begin{equation*}
	\lambda(u,v) = \sum_{i,j = 1}^d \int_\cD a_{i,j} D^i u D^j v \dd x + \int_\cD c u v \dd x,
	\end{equation*}
	where $D^{j}$ denotes weak differentiation with respect to $x_j$, $j = 1, \ldots, d$. The coefficients $a_{i,j}$, $i,j = 1, \ldots, d$, are $\cC^1(\bar{\cD})$ functions  fulfilling $a_{i,j} = a_{j,i}$. Moreover,  we assume that there is a constant $\lambda_0 > 0$ such that for all $y \in \R^d$ and almost all $x \in \cD$, $\sum^{d}_{i,j=1} a_{i,j} (x) y_i y_j \ge \lambda_0 |y|^2$. The function $c \in L^\infty(\cD)$ is non-negative almost everywhere on $\cD$. 
	
	We shall consider two cases of boundary conditions for $\lambda$. In the first case (Dirichlet boundary conditions) we take $V = H^1_0(\cD) = \{v \in H^1 : \gamma v = 0 \}$. Here $\gamma \colon H^r \to L^2(\partial \cD)$ denotes the trace operator, an extension of the mapping $v \mapsto v|_{\partial \cD}$. It is well-defined on $H^r$ for $r > 1/2$ \cite[Theorem~1.3.9]{Y10}.
	In the second case (Neumann boundary conditions) we take $V = H^1$, and assume in addition that there is a constant $c_0 > 0$ such that $c(x) \ge c_0$ almost everywhere on $\cD$. Then $\lambda$ is a symmetric, bilinear form on $V$. The operator $\Lambda$ is regarded as a realization of the strongly elliptic operator 
	\begin{equation*}
	-\sum_{i,j = 1}^d D^j \left(a_{i,j} D^i\right) + c
	\end{equation*}
	on $H$, with boundary conditions $\gamma v = 0$ on $\partial \cD$ in the first case and ${\partial v}/{\partial \nu_\Lambda} = 0 \text{ on } \partial \cD $ in the second case. Here
	\begin{equation*}
	\frac{\partial v}{\partial \nu_\Lambda} = \sum_{i,j = 1}^d n_i a_{i,j} \gamma D^{j} v,
	\end{equation*}
	with $(n_1, \ldots, n_d)$ being the outward unit normal to $\partial \cD$, is well-defined as an element in $L^2(\partial \cD)$ for $v \in H^r$, $r>3/2$. Since the embedding of $V$ into $L^2(\cD)$ is compact and $\Lambda^{-1}$ maps into $V$, the assumption that this operator is compact is true.
	
	Next, we relate the spaces $(\dot{H}^{s})_{s \in [0,2]}$ to $(H^{s})_{s \in [0,2]}$. In the case of Dirichlet boundary conditions, we have
	\begin{equation}
	\label{eq:sobolev_id_1}
	\dot{H}^s = 
	\begin{cases}
	H^s & \text{ if } s \in [0,1/2), \\
	\left\{u \in H^s : \gamma u  = 0 \right\} & \text{ if } s \in (1/2,3/2) \cup (3/2,2],
	\end{cases}	
	\end{equation}
	with norm equivalence \cite[Theorem~16.13]{Y10}. In the case of Neumann boundary conditions, 
	\begin{equation}
	\label{eq:sobolev_id_2}
	\dot{H}^s = 
	\begin{cases}
	H^s & \text{ if } s \in [0, 3/2), \\
	\left\{u \in H^s : {\partial u}/{\partial \nu_\Lambda} = 0 \right\} &\text{ if } s \in (3/2,2],
	\end{cases}	
	\end{equation}
	with norm equivalence \cite[Theorem~16.9]{Y10}.  Note that since $\dot{H}^{-s} = (\dot{H}^s)'$ and ${H}^{-s} = ({H}^s)'$ for $s\ge 0$, analogous results hold for negative exponents $s \in (-1/2,0]$ and $s \in (-3/2,0]$, respectively.
	
	For the identity~\eqref{eq:sobolev_id_2} with $s \in [0,1]$ it suffices to assume that $\partial \cD$ is Lipschitz \cite[Theorem~16.6]{Y10}. Moreover, for both~\eqref{eq:sobolev_id_1} and~\eqref{eq:sobolev_id_2}, one could replace convexity of $\cD$ with $\cC^2$ regularity of $\partial \cD$. In this case, the equivalence in~\eqref{eq:sobolev_id_1} holds also for $s=3/2$ \cite[Theorems~16.7, 16.12]{Y10}.
	
	\subsection{H\"older spaces on bounded domains}
	
	We now introduce the notation we use for H\"older spaces on $\bar{\cD} = \cD \cup \partial \cD$. We write $\cC(\bar{\cD})$ for the space of continuous functions on $\bar{\cD}$ with the supremum norm and $\cC^{0,\sigma}(\bar{\cD}) \subset \cC(\bar{\cD})$ for the space of bounded H\"older continuous functions $f \colon \bar{\cD} \to \R$ with H\"older exponent $\sigma \in (0,1]$. 
	It is equipped with the norm
	\begin{equation*}
	\norm{f}{\cC^{0,\sigma}(\bar{\cD})} = \norm{f}{\cC(\bar{\cD})} + \sup_{\substack{x,y \in \bar{\cD} \\ x \neq y}} \frac{|f(x)-f(y)|}{|x-y|^\sigma}.
	\end{equation*}
	For $k \in \N$, the space $\cC^{k}(\bar{\cD})$ consists of all functions $f \colon \bar{\cD} \to \R$ such that $\partial^\alpha f \in \cC(\bar{\cD})$ for all $|\alpha| = \alpha_1 + \cdots \alpha_d \le k$ while for $\sigma \in (0,1]$, the space $\cC^{k,\sigma}(\bar{\cD})$ consists of all functions $f \in \cC^{k}(\bar{\cD})$ such that $\partial^\alpha f \in \cC^{0,\sigma}(\bar{\cD})$ for all $|\alpha| = k$. Here $\partial^\alpha$ is the classical derivative with respect to a multiindex $\alpha=(\alpha_1,\dots,\alpha_d)$
	and $\partial^\alpha f \in \cC(\bar{\cD})$ means that $\partial^\alpha f \in \cC(\cD)$ and extends to a continuous function on $\bar{\cD}$. 
	We equip $\cC^{k}(\bar{\cD})$ with the norm 
	\begin{equation*}
	\norm{f}{\cC^{k,\sigma}(\bar{\cD})} = \sum_{\alpha \le k } \norm{\partial^\alpha f}{\cC(\bar{\cD})}
	\end{equation*}
	and $\cC^{k,\sigma}(\bar{\cD})$ with the norm
	\begin{equation*}
	\norm{f}{\cC^{k}(\bar{\cD})} = \sum_{\alpha < k } \norm{\partial^\alpha f}{\cC(\bar{\cD})} + \sum_{\alpha = k } \norm{\partial^\alpha f}{\cC^{0,\sigma}(\bar{\cD})}.
	\end{equation*}
	We will also need the space $\cC^{k,k}(\bar{\cD} \times \bar{\cD})$, which we define as the space of continuous functions $f \colon \bar{\cD} \times \bar{\cD} \to \R$ such that $\partial_1^{\alpha} \partial_2^{\alpha} f \in \cC(\bar{\cD} \times \bar{\cD})$ for all $|\alpha| \le k$. The space $\cC^{k,k,\sigma}(\bar{\cD} \times \bar{\cD})$ consists of all functions $f \in \cC^{k,k}(\bar{\cD} \times \bar{\cD})$ such that $\partial_1^{\alpha} \partial_2^{\alpha} f \in \cC^{0,\sigma}(\bar{\cD} \times \bar{\cD})$ for all $|\alpha| = k$. Here $\partial_i^\alpha$ denotes differentiation with respect to the $i$th variable of $f$. 
	We equip $\cC^{k,k}(\bar{\cD}\times\bar{\cD})$ with the norm 
	\begin{equation*}
	\norm{f}{\cC^{k,k}(\bar{\cD} \times \bar{\cD})} =  \sum_{\alpha \le k } \norm{\partial_1^\alpha \partial_2^\alpha f}{\cC(\bar{\cD}\times\bar{\cD})}
	\end{equation*}
	and $\cC^{k,k,\sigma}(\bar{\cD}\times\bar{\cD})$ with the norm 
	\begin{equation*}
	\norm{f}{\cC^{k,k,\sigma}(\bar{\cD} \times \bar{\cD})} =  \sum_{\alpha < k } \norm{\partial_1^\alpha \partial_2^\alpha f}{\cC(\bar{\cD}\times\bar{\cD})} + \sum_{\alpha = k } \norm{\partial_1^\alpha \partial_2^\alpha f}{\cC^{0,\sigma}(\bar{\cD}\times\bar{\cD})}.
	\end{equation*}
	
	\subsection{Integral operators and reproducing kernel Hilbert spaces}
	
	Let $q \colon \bar{\cD} \times \bar{\cD} \to \R$ be a positive semidefinite symmetric continuous kernel. We define an  operator $Q \in \cL(H)$ by
	\begin{equation*}
	Q u (x) = \int_{\cD} q(x,y) u(y) \dd y
	\end{equation*}
	for $x \in \bar{\cD}$.
	
	By Mercer's theorem \cite[Theorem~A.8]{PZ07}, $Q$ is then self-adjoint, positive semidefinite and of trace class. Moreover, $Q$ admits an orthonormal eigenbasis $(q_j)_{j =1}^\infty$ with a corresponding non-increasing sequence of non-negative eigenvalues $(\mu_j)_{j =1}^\infty$. The eigenfunctions corresponding to nonzero eigenvalues are continuous. The kernel can then, with $x,y \in \bar{\cD}$, be represented by 
	\begin{equation}
	\label{eq:kernelexpansion}
	q(x,y) = \sum_{j=1}^{\infty} \mu_j q_j(x) q_j(y),
	\end{equation}
	where the sequence converges uniformly and absolutely on $\bar{\cD} \times \bar{\cD}$. 
	
	We write $Q^{-1/2} = (Q^{1/2}|_{\ker(Q^{1/2})^\perp})^{-1} \colon Q^{1/2}(H) \to \ker(Q^{1/2})^\perp \subset H$ for the pseudo-inverse of $Q^{1/2}$. The space $Q^{1/2}(H)$, the range of $Q^{1/2}$ in $H$, is then a Hilbert space with respect to the inner product $\inpro[Q^{1/2}(H)]{\cdot}{\cdot} = \inpro[H]{Q^{-1/2}\cdot}{Q^{-1/2}\cdot}$. Under the notational convention $0/0=0$, we note that, in light of Mercer's theorem,
	\begin{equation}
	\label{eq:mercer-qhalf-rep}
	Q^{\frac{1}{2}}(H) = \left\{f \in H: \sum_{j=1}^\infty \frac{|\inpro[H]{f}{q_j}|^2}{\mu_j} < \infty \right\}.
	\end{equation}
	By the same theorem, it follows that for all $f \in H, N \in \N$,
	\begin{equation*}\Big|\sum_{j = 1}^N \inpro[H]{Q^{1/2}f}{ q_j} q_j(x)\Big|^2 
	\le \Big(\sum_{j = 1}^N \mu_j q_j(x)^2\Big) \Big(\sum_{j = 1}^N |\inpro[H]{f}{q_j}|^2\Big)  \le q(x,x) \norm{f}{H}^2.
	\end{equation*}
	Therefore, the elements in $Q^{\frac{1}{2}}(H)$ are continuous functions on $\bar{\cD}$, as opposed to just equivalence classes of functions on $\cD$. 
	
	The kernel $q$ is said to satisfy the Dirichlet or Neumann boundary conditions if 
	\begin{equation}
	\label{eq:QDirichlet}
	q(\cdot,y) = q(y,\cdot) = 0
	\end{equation}
	or
	\begin{equation}
	\label{eq:QNeumann}
	\sum_{i,j = 1}^d n_i(\cdot) a_{i,j}(\cdot) \frac{\partial q(\cdot,y)}{\partial x_j} = \sum_{i,j = 1}^d n_i(\cdot) a_{i,j}(\cdot) \frac{\partial q(y,\cdot)}{\partial y_j} = 0
	\end{equation}
	on $\partial \cD$, respectively, for all $y \in \cD$. Note that, provided that $q$ is sufficiently smooth, this guarantees that $Q u$ fulfills the corresponding boundary conditions for arbitrary $u \in H$. 
	
	We will make use of the theory of reproducing kernel Hilbert spaces. Recall (see, e.g.,~\cite{BT04}) that a Hilbert space $H_q(\cE)$ is said to be the reproducing kernel Hilbert space of the kernel $q$ if it is a Hilbert space of functions on a non-empty subset $\cE \subseteq \R^d$ such that the conditions 
	\begin{enumerate}[label=(\roman*)]
		\item \label{RKHSprop1} $q(x,\cdot) \in H_q(\cE)$ for all $x \in \cE$ and
		\item \label{RKHSprop2} for all $f \in H_q(\cE), x \in \cE$, $f(x)=\inpro[H_q(\cE)]{f}{q(x,\cdot)}$
	\end{enumerate}
	are satisfied.
	The Moore--Aronszajn theorem (see~\cite[Theorem~3]{BT04}) can be used to construct $H_q(\cE)$ explicitly. Given a (not necessarily continuous) positive semidefinite function $q$, we write $H_{q,0}(\cE)$ for the linear span of functions of the type $f = \sum_{i = 1}^n \alpha_i q(x_i,\cdot)$, $(x_i)_{i=1}^n \subset \cE$, $(\alpha_i)_{i=1}^n \subset \R$, $n \in \N$. We define a symmetric positive semidefinite bilinear form on $H_{q,0}(\cE)$ by
	\begin{equation*}
	\inpro[H_{q}(\cE)]{f}{g} = \sum_{i = 1}^n \sum_{j = 1}^m \alpha_{i} \beta_j q(x_i,y_j),
	\end{equation*}
	for $f$ as above, $g = \sum_{j = 1}^m \beta_j q(y_j,\cdot) \in H_{q,0}(\cE)$, $(y_j)_{j=1}^m \subset \cE$ and $(\beta_j)_{j=1}^m \subset \R$. By the Cauchy--Schwarz inequality for such forms, $|f(x)| = | \inpro[H_{q}(\cE)]{f}{q(\cdot,x)} | \le \inpro[H_q(\cE)]{f}{f}^{1/2} q(x,x)^{1/2}$ for all $x \in \cE$ and $f \in H_{q,0}(\cE)$.
	This shows that $\inpro[H_q(\cE)]{\cdot}{\cdot}$ is an inner product on $H_{q,0}(\cE)$. The space $H_q(\cE)$ is now given as the set of functions $f$ on $\cE$ for which there is a Cauchy sequence $(f_j)_{j=1}^\infty$ in $H_{q,0}(\cE)$ converging pointwise to $f$. It is a Hilbert space with inner product $\inpro[H_q(\cE)]{f}{g} = \lim_{n\to\infty} \inpro[H_q(\cE)]{f_n}{g_n}$ and it fulfills the conditions~\ref{RKHSprop1} and~\ref{RKHSprop2}. In fact, $H_q(\cE)$ is the unique Hilbert space of functions on $\cE$ fulfilling these conditions. From here on, we use the shorthand notation $H_q = H_q(\bar{\cD})$. 
	
	The property~\ref{RKHSprop2} is referred to as the reproducing kernel property of the space $H_q$. When $q$ is sufficiently smooth, then an analogous property holds for the derivatives of $f \in H_q$. To be precise, if $q \in \cC^{k,k}(\bar{\cD} \times \bar{\cD})$, then $H_q \hookrightarrow \cC^{k}(\bar{\cD})$ continuously and
	\begin{equation}
	\label{eq:RKHSderivprop} 
	\partial^\alpha f(x) = \partial^\alpha \inpro[H_q]{f}{ q(x,\cdot)} =  \inpro[H_q]{f}{\partial_1^\alpha q(x,\cdot)} 
	\end{equation}
	for all $f \in H_q, x \in \bar{\cD}$ and $|\alpha| \le k$. Moreover, the symmetric function $\partial_1^\alpha \partial_2^\alpha q$ possesses the reproducing property
	\begin{equation}
	\label{eq:RKHSderivprop2}
	\partial_1^\alpha \partial_2^\alpha q(y,x) = \inpro[H_q]{ \partial_1^\alpha q(y,\cdot)}{\partial_1^\alpha q(x,\cdot)},
	\end{equation}
	for $x,y \in \bar{\cD}$ and $|\alpha| \le k$. These facts are proven in \cite{Z08} under the stronger condition $q \in \cC^{2k}(\bar{\cD} \times \bar{\cD})$ (defined analogously to $\cC^{2k}(\bar{\cD}$)) but the proof works also in the case of our weaker condition, cf.~\cite[Corollary~4.36]{SC08} for the case that $q \in \cC^{k,k}(\cD \times \cD)$.
	
	A key observation that we will make use of is the fact that, under our assumptions on $q$, $Q^{{1/2}}(H) = H_q$. This can be seen by explicitly checking that~\ref{RKHSprop1} and~\ref{RKHSprop2} are fulfilled by $Q^{{1/2}}(H)$, using the expressions~\eqref{eq:kernelexpansion} and~\eqref{eq:mercer-qhalf-rep}. We refer to~\cite[Theorem~10.29]{W04} and~\cite{W70} for analogous arguments under somewhat different assumptions on $q$. 
	
	In the next two sections, we will identify conditions on $q$ which ensure that $Q \in \cL_p(\dot{H}^{s},\dot{H}^r)$ for appropriate powers $r,s \in \R$ and $p \ge 1$. We first consider the non-homogeneous case, which is to say that $q(x,y)$ depends explicitly on $x,y$, not necessarily only on $x-y$ or $|x-y|$ (the homogeneous and isotropic cases, respectively). 
	
	\section{Regularity of integral operators in the non-homogeneous case}
	\label{sec:non-hom}
	
	In this section we consider H\"older regularity assumptions on $q$. We start by deriving conditions for which the trace class condition~\eqref{eq:tracecond} holds. We recall that this is given by
	\begin{equation*}
	\norm{\Lambda^{\frac{r}{2}}Q\Lambda^{\frac{r}{2}}}{\cL_1(H)} = \norm{\Lambda^{\frac{r}{2}}Q^{\frac{1}{2}}}{\cL_2(H)}^2 < \infty
	\end{equation*}
	with $r\ge0$. We use reproducing kernel Hilbert spaces in the proof of the following two theorems, corresponding to the two types of boundary conditions considered. Below we write ''for all $|\alpha|=1$'' as shorthand for ''for all multiindices $\alpha = (\alpha_1, \ldots, \alpha_d)$ such that $|\alpha|=1$''.
	
	\begin{theorem}
		\label{thm:nonhomtracedirichlet}
		Let $\Lambda$ be the elliptic operator associated with the bilinear form $\lambda$ with Dirichlet boundary conditions. Then, the operator $Q$ with kernel $q \in \cC(\bar{\cD}\times\bar{\cD})$ satisfies~\eqref{eq:tracecond}
		\begin{enumerate}[label=(\roman*)]
			\item %
			for all $r \in [0,\sigma/2)$ if there exists $\sigma \in (0,1]$ such that for a.e.\ $z \in \cD$, $q(z,\cdot) \in C^{0,\sigma}(\bar{\cD})$ with $\esssup_{z \in \cD} \norm{q(z,\cdot)}{C^{0,\sigma}(\bar{\cD})} < \infty$.
		\end{enumerate}
		If $q$ in addition satisfies the boundary conditions~\eqref{eq:QDirichlet}, then $Q$ satisfies~\eqref{eq:tracecond}
		\begin{enumerate}[label=(\roman*),resume]
			\item \label{thm:nonhomtracedirichlet:ii} for all $r \in [0,(1+\sigma)/2) \setminus \{1/2\}$ if $q$ is once differentiable in the first variable, and there is a constant $\sigma \in (0,1]$ such that for all $|\alpha|=1$ and a.e.\ $z \in \cD$, $\partial_1^{\alpha} q(z,\cdot) \in C^{0,\sigma}(\bar{\cD})$ with $\esssup_{z \in \cD}\norm{\partial_1^{\alpha} q(z,\cdot)}{C^{0,\sigma}(\bar{\cD})} < \infty$,
			\item \label{thm:nonhomtracedirichlet:iii:new} for all $r\in [0,1]\setminus \{1/2\}$ if $q \in C^{1,1}(\bar{\cD}\times\bar{\cD})$,
			\item \label{thm:nonhomtracedirichlet:iii} for all $r\in [0,(2+\sigma)/2)\setminus \{1/2\}$ if there exists $\sigma \in (0,1]$ such that $\norm{q}{C^{1,1,\sigma}(\bar{\cD}\times\bar{\cD})} < \infty$,
			\item \label{thm:nonhomtracedirichlet:iv} for all $r \in [0,(3+\sigma)/2) \setminus \{1/2,3/2\}$ if $q \in C^{1,1}(\bar{\cD}\times\bar{\cD})$ and, for all $|\alpha| = 1$, $\partial_1^\alpha \partial_2^\alpha q$ is once differentiable in the first variable and there exists $\sigma \in (0,1]$ such that for all $|\alpha'|=1$ and a.e.\ $z \in \cD$, $\partial_1^{\alpha+\alpha'} \partial_2^\alpha q(z,\cdot) \in C^{0,\sigma}(\bar{\cD})$ with $\esssup_{z \in \cD}\norm{\partial_1^{\alpha+\alpha'} \partial_2^\alpha q(z,\cdot)}{C^{0,\sigma}(\bar{\cD})} < \infty$,
			\item \label{thm:nonhomtracedirichlet:iv:temp} for all $r \in [0,2] \setminus \{1/2,3/2\}$ if $q \in C^{2,2}(\bar{\cD}\times\bar{\cD})$.
		\end{enumerate}
	\end{theorem}
	\begin{proof}
		Let $(e_j)_{j=1}^\infty$ be the eigenbasis of $\Lambda$. We note that $(Q^{{1/2}}e_j)_{j=1}^\infty$ is then an orthonormal basis of $Q^{{1/2}}(H) = H_q$. For the first statement, by this observation, the definition of the Hilbert--Schmidt norm, \eqref{eq:sobolev_id_2}, and the reproducing kernel property of $H_q$, we have
		\begin{align*}
		\norm{\Lambda^{\frac{r}{2}}Q^{\frac{1}{2}}}{\cL_2(H)}^2 &= \sum_{j = 1}^{\infty} \norm{Q^{\frac{1}{2}} e_j}{\dot{H}^r}^2 \lesssim \sum_{j = 1}^{\infty} \norm{Q^{\frac{1}{2}} e_j}{H^r}^2 \\ &= \sum_{j = 1}^{\infty} \norm{Q^{\frac{1}{2}} e_j}{H}^2 + \int_{\cD \times \cD} \frac{\sum_{j = 1}^{\infty} |Q^{\frac{1}{2}}e_j(x)-Q^{\frac{1}{2}}e_j(y)|^2}{|x-y|^{d+2r}} \dd x \dd y \\ &= \trace(Q) + \int_{\cD \times \cD} \frac{\sum_{j = 1}^{\infty} \left|\inpro[H_q]{Q^{\frac{1}{2}} e_j}{q(x,\cdot) - q(y,\cdot) }\right|^2 }{|x-y|^{d+2r}} \dd x \dd y \\
		&= \trace(Q) + \int_{\cD \times \cD} \frac{ \norm{q(x,\cdot) - q(y,\cdot)}{H_q}^2 }{|x-y|^{d+2r}} \dd x \dd y \\  
		&= \trace(Q) + \int_{\cD \times \cD} \frac{ (q(x,x)-q(x,y))+(q(y,y)-q(y,x)) }{|x-y|^{d+2r}} \dd x \dd y \\
		&\le \trace(Q) + 2\esssup_{z \in \cD} \norm{q(z,\cdot)}{C^{0,\sigma}(\bar{\cD})} \int_{\cD \times \cD} |x-y|^{-d+\sigma-2r} \dd x \dd y < \infty.
		\end{align*}
		The integral is finite since $\sigma-2r>0$. For case \ref{thm:nonhomtracedirichlet:ii}, we note that $Q^{{1/2}} e_j(x) = \inpro[H_q]{Q^{{1/2}} e_j}{q(x,\cdot)}$ for all $x \in \partial\cD$, so that $Q^{{1/2}} e_j$ inherits the boundary conditions of $q$. The proof is now the same as before, except that we use the mean value theorem 
		to deduce that 
		\begin{align*}
		\norm{q(x,\cdot) - q(y,\cdot)}{H_q}^2 &= (q(x,x)-q(y,x))-(q(x,y)-q(y,y)) \\
		&= \int_{0}^{1} \big(\nabla_1 q((1-s)x+sy,x) - \nabla_1 q((1-s)x+sy,y)\big)\cdot(x-y) \dd s \\
		&\le |x-y| \sum_{j=1}^d \int^1_0 \left| \partial^{x_j}_1 q((1-s)x+sy,x) - \partial^{x_j}_1 q((1-s)x+sy,y) \right| \dd s \\
		&\le d |x-y|^{1+\sigma} \max_{j = 1, \dots, d} \esssup_{z \in \cD} \norm{\partial^{z_j}_1 q(z,\cdot)}{\cC^{0,\sigma}(\bar{\cD})}.
		\end{align*}
		Here $\nabla_1 q = (\partial^{x_1}_1 q, \ldots, \partial^{x_d}_1 q)$ is the gradient with respect to the first component of $q$ and $\partial^{x_j}_1q(x,y)$ is the derivative of $q(x,y) = q(x_1, \dots, x_d,y_1, \dots, y_d)$ with respect to $x_j$.
		This yields
		\begin{align*}
		\norm{\Lambda^{\frac{r}{2}}Q^{\frac{1}{2}}}{\cL_2(H)}^2 \lesssim \trace(Q) + \max_{j = 1, \dots, d} \esssup_{z \in \cD} \norm{\partial^{z_j}_1 q(z,\cdot)}{\cC^{0,\sigma}(\bar{\cD})} \int_{\cD \times \cD} |x-y|^{-d+1+\sigma-2r} \dd x \dd y < \infty.
		\end{align*}
		For case~\ref{thm:nonhomtracedirichlet:iii:new}, we also apply \eqref{eq:RKHSderivprop} and~\eqref{eq:RKHSderivprop2} to see that 
		\begin{align*}
		\norm{\Lambda^{\frac{1}{2}}Q^{\frac{1}{2}}}{\cL_2(H)}^2 \lesssim\sum_{|\alpha|\le 1} \sum_{j = 1}^{\infty} \norm{\partial^{\alpha}Q^{\frac{1}{2}} e_j}{H}^2 
		&= \sum_{|\alpha|\le 1} \sum_{j = 1}^{\infty} \int_{\cD} |\partial^\alpha Q^{\frac{1}{2}} e_j(x)|^2 \dd x \\
		&= \sum_{|\alpha|\le 1} \sum_{j = 1}^{\infty} \int_{\cD} \left|\inpro[H_q]{Q^{\frac{1}{2}} e_j}{\partial_1^\alpha q(x,\cdot)}\right|^2 \dd x \\
		&= \sum_{|\alpha|\le 1} \int_{\cD} \norm{\partial_1^\alpha q(x,\cdot)}{H_q}^2 \dd x = \sum_{|\alpha|\le 1} \int_{\cD} \partial_1^\alpha \partial_2^\alpha q(x,x)\dd x,
		\end{align*}
		which is bounded by a constant times $\norm{q}{C^{1,1}(\bar{\cD}\times\bar{\cD})}$.
		For case~\ref{thm:nonhomtracedirichlet:iii}, we obtain, for $r>1$,
		\begin{align*}
		&\sum_{|\alpha|= 1}  \int_{\cD \times \cD} \frac{\sum_{j = 1}^{\infty} |\partial^{\alpha}Q^{\frac{1}{2}}e_j(x)-\partial^{\alpha}Q^{\frac{1}{2}}e_j(y)|^2}{|x-y|^{d+2(r-1)}} \dd x \dd y \\
		&\quad= \sum_{|\alpha|= 1}\int_{\cD \times \cD} \frac{\sum_{j = 1}^{\infty} \left|\inpro[H_q]{Q^{\frac{1}{2}} e_j}{\partial_1^{\alpha} q(x,\cdot) - \partial_1^{\alpha}q(y,\cdot) }\right|^2 }{|x-y|^{d+2(r-1)}} \dd x \dd y \\
		&\quad= \sum_{|\alpha|= 1}\int_{\cD \times \cD} \frac{ (\partial_1^{\alpha} \partial_2^{\alpha} q(x,x)-\partial_1^{\alpha} \partial_2^{\alpha} q(x,y))+(\partial_1^{\alpha} \partial_2^{\alpha}q(y,y)-\partial_1^{\alpha} \partial_2^{\alpha}q(y,x)) }{|x-y|^{d+2(r-1)}} \dd x \dd y 
		\end{align*}
		so that
		\begin{align*}
		\norm{\Lambda^{\frac{r}{2}}Q^{\frac{1}{2}}}{\cL_2(H)}^2 &\lesssim \sum_{|\alpha|\le 1} \sum_{j = 1}^{\infty} \norm{\partial^{\alpha}Q^{\frac{1}{2}} e_j}{H}^2 + \sum_{|\alpha|= 1}  \int_{\cD \times \cD} \frac{\sum_{j = 1}^{\infty} |\partial^{\alpha}Q^{\frac{1}{2}}e_j(x)-\partial^{\alpha}Q^{\frac{1}{2}}e_j(y)|^2}{|x-y|^{d+2(r-1)}} \dd x \dd y \\
		&\lesssim \norm{q}{C^{1,1}(\bar{\cD}\times\bar{\cD})} + \norm{q}{C^{1,1,\sigma}(\bar{\cD}\times\bar{\cD})} \int_{\cD \times \cD} |x-y|^{-d+2+\sigma-2r} \dd x \dd y < \infty.
		\end{align*}
		Finally, case \ref{thm:nonhomtracedirichlet:iv} is obtained by a modification of this argument, similar to how case \ref{thm:nonhomtracedirichlet:ii} was obtained, while the proof of case~\ref{thm:nonhomtracedirichlet:iv:temp} is analogous to that of case~\ref{thm:nonhomtracedirichlet:iii:new}.
	\end{proof}
	
	\begin{remark}
		\label{rem:bridge-1}
		This result is sharp, in the following sense. Consider the setting of $\cD = (0,1)$ and $\Lambda = (-\Delta)$, where $\Delta$ is equipped with zero boundary conditions. In this setting, the kernel of the operator $Q = \Lambda^{-1}$ is explicitly given by $q(x,y) = \min(x,y)- xy;$ for $x,y \in \cD$, see, e.g., \cite{CFM15}. By Theorem~\ref{thm:nonhomtracedirichlet}, \eqref{eq:tracecond} is satisfied for $r<1/2$. The eigenpairs associated to $Q$ are given by $q_j(x) = \sqrt{2} \sin(\pi j x)$ and $\mu_j = (\pi j)^{-2}$ for $j \in \N$. Therefore, the expression in~\eqref{eq:tracecond} is infinite for $r \ge 1/2$.
	\end{remark}
	
	\begin{theorem}
		\label{thm:nonhomtraceneumann}
		Let $\Lambda$ be the elliptic operator associated with the bilinear form $\lambda$ with Neumann boundary conditions. Then, the operator $Q$ with kernel $q \in \cC(\bar{\cD}\times\bar{\cD})$ satisfies~\eqref{eq:tracecond}
		\begin{enumerate}[label=(\roman*)]
			\item for all $r \in [0,\sigma/2)$ if there exists $\sigma \in (0,1]$ such that for a.e.\ $z \in \cD$, $q(z,\cdot) \in C^{0,\sigma}(\bar{\cD})$ with $\esssup_{z \in \cD} \norm{q(z,\cdot)}{C^{0,\sigma}(\bar{\cD})} < \infty$,
			\item for all $r \in [0,(1+\sigma)/2)$ if $q$ is once differentiable in the first variable, and there is a constant $\sigma \in (0,1]$ such that for all $|\alpha|=1$ and a.e.\ $z \in \cD$, $\partial_1^{\alpha} q(z,\cdot) \in C^{0,\sigma}(\bar{\cD})$ with $\esssup_{z \in \cD}\norm{\partial_1^{\alpha} q(z,\cdot)}{C^{0,\sigma}(\bar{\cD})} < \infty$,
			\item for all $r\in [0,1]\setminus \{1/2\}$ if $q \in C^{1,1}(\bar{\cD}\times\bar{\cD})$,
			\item for all $r\in [0,(2+\sigma)/2)$ if there exists $\sigma \in (0,1]$ such that $\norm{q}{C^{1,1,\sigma}(\bar{\cD}\times\bar{\cD})} < \infty$.
		\end{enumerate}
		If $q$ additionally satisfies the boundary conditions~\eqref{eq:QNeumann}, then $Q$ satisfies~\eqref{eq:tracecond}
		\begin{enumerate}[label=(\roman*),resume]		 
			\item  for all $r \in [0,(3+\sigma)/2) \setminus \{3/2\}$ if $q \in C^{1,1}(\bar{\cD}\times\bar{\cD})$ and, for all $|\alpha| = 1$, $\partial^\alpha \partial^\alpha q$ is once differentiable in the first variable and there exists $\sigma \in (0,1]$ such that for all $|\alpha'|=1$ and a.e.\ $z \in \cD$, $\partial_1^{\alpha+\alpha'} \partial_2^\alpha q(z,\cdot) \in C^{0,\sigma}(\bar{\cD})$ with $\esssup_{z \in \cD}\norm{\partial_1^{\alpha+\alpha'} \partial_2^\alpha q(z,\cdot)}{C^{0,\sigma}(\bar{\cD})} < \infty$,
			\item for all $r \in [0,2] \setminus \{1/2,3/2\}$ if $q \in C^{2,2}(\bar{\cD}\times\bar{\cD})$.
		\end{enumerate}
	\end{theorem}
	\begin{proof}
		The proof is the same as in the previous theorem. We simply note that for the last case, since $q \in \cC^{1,1}(\bar{\cD}\times\bar{\cD})$, the boundary condition~\eqref{eq:QNeumann} is well-defined. Moreover, since 
		\begin{equation*}
		\partial^\alpha Q^{\frac{1}{2}} e_j(x) = \inpro[H_q]{Q^{\frac{1}{2}} e_j}{\partial_1^\alpha q(x,\cdot)} 
		\end{equation*}
		for all $x \in \bar{\cD}$, $Q^{{1/2}} e_j$ inherits the boundary conditions of $q$.
	\end{proof}
	
	In the next two theorems we derive conditions on the kernel $q$ that guarantee that the estimate~\eqref{eq:Qreg} is satisfied for appropriate powers $r,s \ge 0$. We recall that it is given by
	\begin{equation*}
	\norm{\Lambda^{\frac{r}{2}}Q\Lambda^{\frac{s}{2}}}{\cL_2(H)} = \norm{\Lambda^{\frac{s}{2}}Q\Lambda^{\frac{r}{2}}}{\cL_2(H)} < \infty.
	\end{equation*}
	In this case we do not rely on the reproducing kernel property but can use more elementary techniques. 
	\begin{theorem}
		\label{prop:Qreg}
		Let $\Lambda$ be the elliptic operator associated with the bilinear form $\lambda$ with Dirichlet boundary conditions. Then, the operator $Q$ with kernel $q \in \cC(\bar{\cD}\times\bar{\cD})$ satisfies\eqref{eq:Qreg}
		\begin{enumerate}[label=(\roman*)]
			\item \label{thm:nonhomhsdirichlet:i} for all $r,s \in [0,1/2)$ such that $r+s < \sigma$ if there exists $\sigma \in (0,1]$ such that for a.e.\ $z \in \cD$, $q(z,\cdot) \in C^{0,\sigma}(\bar{\cD})$ with $\esssup_{z \in \cD} \norm{q(z,\cdot)}{C^{0,\sigma}(\bar{\cD})} < \infty$.
		\end{enumerate}
		If $q$ also satisfies the boundary conditions~\eqref{eq:QDirichlet}, then $Q$ satisfies~\eqref{eq:Qreg}
		\begin{enumerate}[label=(\roman*),resume]
			\item \label{thm:nonhomhsdirichlet:ii} for all $r,s \in [0,2] \setminus \{1/2,3/2\}$ such that $r+s < k + \sigma$ if there exist $k \ge 0$ and $\sigma \in (0,1]$ such that for a.e.\ $z \in \cD$, $q(z,\cdot) \in C^{k,\sigma}(\bar{\cD})$ with $\esssup_{z \in \cD} \norm{q(z,\cdot)}{C^{k,\sigma}(\bar{\cD})} < \infty$.
		\end{enumerate}
		If $\sigma = 1$ in~\ref{thm:nonhomhsdirichlet:ii} the statement is true with $r+s \le k + 1$.
	\end{theorem}
	\begin{proof}
		Since $q$ is symmetric, we have $\esssup_{z \in \cD} \norm{q(z,\cdot)}{C^{k,\sigma}(\bar{\cD})} = \esssup_{z \in \cD} \norm{q(\cdot,z)}{C^{k,\sigma}(\bar{\cD})}$  for $k, \sigma \ge 0$. Moreover, because of~\eqref{eq:sobolev_id_1} and Lemmata~\ref{lem:symmetricschatten} and~\ref{lem:symmetricschatten2}, it suffices to show that if there exist $k \ge 0$ and $\sigma \in (0,1]$ such that for a.e.\ $z \in \cD$, $q(z,\cdot) \in C^{k,\sigma}(\bar{\cD})$ with $\esssup_{z \in \cD} \norm{q(z,\cdot)}{C^{k,\sigma}(\bar{\cD})} < \infty$, then $Q \in \cL_2(H,H^r)$ for all $r < k + \sigma$, and for all $r \le k + 1$ when $\sigma = 1$. Under this condition, we have by definition of the weak derivative that
		\begin{equation*}
		D^\alpha \int_{\cD} q(\cdot,y) e_j(y) \dd y = \int_{\cD} D^\alpha_1 q(\cdot,y) e_j(y) \dd y = \int_{\cD} \partial^\alpha_1 q(\cdot,y) e_j(y) \dd y
		\end{equation*}
		for $|\alpha| \le k$, where $(e_j)_{j=1}^\infty$ is an orthonormal basis of~$H$. This is also true when $q(z,\cdot) \in C^{k-1,1}(\bar{\cD})$, $k \ge 1$, with $\esssup_{z \in \cD} \norm{q(z,\cdot)}{C^{k-1,1}(\bar{\cD})} < \infty$, since then $\partial^{\alpha}_1 q(\cdot,y)$, $|\alpha| \le k-1$, has a bounded classical derivative almost everywhere in~$\cD$. Under either of these conditions, we obtain from the definition of the Hilbert--Schmidt norm that
		\begin{align*}
		\norm{Q}{\cL_2(H,H^k)}^2 = \sum_{j=1}^{\infty} \norm{Qe_j}{H^k}^2 &= \sum_{j=1}^{\infty} \sum_{|\alpha| \le k}\Bignorm{ D^\alpha  \int_{\cD} q(\cdot,y) e_j(y) \dd y}{H}^2 \\ &= \sum_{|\alpha| \le k} \int_{\cD} \sum_{j=1}^{\infty} \left| \int_{\cD} \partial^\alpha_1 q(x,y) e_j(y) \dd y \right|^2 \dd x \\ &= \sum_{|\alpha| \le k} \int_{\cD} \norm{\partial^\alpha_1 q(x,\cdot)}{H}^2 \dd x \\
		&= \sum_{|\alpha| \le k} \int_{\cD} \int_{\cD} |\partial^\alpha_1 q(x,y)|^2 \dd x \dd y  < \infty,
		\end{align*}
		which finishes the proof of the very last statement of the theorem. Next we have, for $k < r < k + \sigma$,
		\begin{align*}
		\norm{Q}{\cL_2(H,H^r)}^2 
		&= \sum_{j = 1}^{\infty} \norm{Q e_j}{H^{k}}^2 + \sum_{|\alpha| = k} \int_{\cD \times \cD} \frac{\sum_{j = 1}^{\infty} |D^\alpha Qe_j(x)-D^\alpha Qe_j(y)|^2}{|x-y|^{d+2(r-k)}} \dd x \dd y \\ &= \norm{Q}{\cL_2(H,H^k)}^2 + \sum_{|\alpha| = k} \int_{\cD \times \cD} \frac{\sum_{j = 1}^{\infty} \left|\inpro[H]{\partial_1^\alpha q(x,\cdot) - \partial_1^\alpha q(y,\cdot) }{e_j}\right|^2 }{|x-y|^{d+2(r-k)}} \dd x \dd y \\
		&= \norm{Q}{\cL_2(H,H^k)}^2 + \sum_{|\alpha| = k}  \int_{\cD \times \cD} \frac{ \norm{\partial_1^\alpha q(x,\cdot) - \partial_1^\alpha q(y,\cdot)}{H}^2 }{|x-y|^{d+2(r-k)}} \dd x \dd y \\  
		&\lesssim \norm{Q}{\cL_2(H,H^k)}^2 + \sum_{|\alpha| = k} \int_{\cD \times \cD} \frac{ \esssup_{z \in \cD} \norm{q(z,\cdot)}{C^{k,\sigma}(\bar{\cD})}^2 }{|x-y|^{d+2(r-\sigma-k)}} \dd x \dd y < \infty,
		\end{align*}
		which completes the proof.
	\end{proof}
	
	\begin{remark}
		\label{rem:bridge-2}
		This result is not sharp for the example of Remark~\ref{rem:bridge-1}. From the explicit representation of the eigenpairs of $Q$, it follows that~\eqref{eq:Qreg} is fulfilled for all $r,s \in [0,3/2)$ such that $r+s < 3/2$. However, Theorem~\ref{prop:Qreg} only guarantees~\eqref{eq:Qreg} to hold for $r,s \in [0,1]$ such that $r+s \le 1$. The proof could in this case be amended to recover the sharp result by a more involved analysis of the term $$\int_{\cD \times \cD} \frac{ \norm{\partial_1 q(x,\cdot) - \partial_1 q(y,\cdot)}{H}^2 }{|x-y|^{1+2(r-k)}} \dd x \dd y,$$ with $k=1$, since in this case the (discontinuous) function $\partial_1 q(x,\cdot)$ is explicitly known. However, as we only consider H\"older conditions in this section, we do not pursue this direction further. 
	\end{remark}
	
	The following theorem for Neumann boundary conditions can be proven in the same way.
	
	\begin{theorem}
		\label{thm:Qnonhomhsneumann}
		Let $\Lambda$ be the elliptic operator associated with the bilinear form $\lambda$ with Neumann boundary conditions. Then, the operator $Q$ with kernel $q \in \cC(\bar{\cD}\times\bar{\cD})$ satisfies~\eqref{eq:Qreg}
		\begin{enumerate}[label=(\roman*)]
			\item %
			for all $r,s \in [0,1]$ such that $r+s < \sigma$ if there exists $\sigma \in (0,1]$ such that for a.e.\ $z \in \cD$, $q(z,\cdot) \in C^{0,\sigma}(\bar{\cD})$ with $\esssup_{z \in \cD} \norm{q(z,\cdot)}{C^{0,\sigma}(\bar{\cD})} < \infty$,		
			\item \label{thm:nonhomhsneumann:ii}
			for all $r,s \in [0,3/2)$ such that $r+s < 1+\sigma$
			if there exists $\sigma \in (0,1]$ such that for a.e.\ $z \in \cD$, $q(z,\cdot) \in C^{1,\sigma}(\bar{\cD})$ with $\esssup_{z \in \cD} \norm{q(z,\cdot)}{C^{1,\sigma}(\bar{\cD})} < \infty$.
		\end{enumerate}
		If $q$ also satisfies the boundary conditions~\eqref{eq:QNeumann}, then $Q$ satisfies~\eqref{eq:Qreg}
		\begin{enumerate}[label=(\roman*),resume]
			\item %
			for all $r,s \in [0,2]\setminus\{3/2\}$ such that $r+s < k+\sigma$ if there exist $k \ge 0$ and $\sigma \in (0,1]$ such that for a.e.\ $z \in \cD$, $q(z,\cdot) \in C^{k,\sigma}(\bar{\cD})$ with $\esssup_{z \in \cD} \norm{q(z,\cdot)}{C^{k,\sigma}(\bar{\cD})} < \infty$.
		\end{enumerate}
		Moreover, if $\sigma = 1$ above, the statements remain true with $r + s \le 1$, $r + s \le 2$ and $r + s \le k + 1$, respectively.
	\end{theorem}
	
	\section{Regularity of integral operators in the homogeneous case}
	\label{sec:hom}
	
	We now move on to the case of a homogeneous (or stationary) kernel, i.e., when $q$ is taken to be a function defined on the unbounded space $\R^d \times \R^d$ of the form $q(x,y) = q(x-y)$ for $x,y \in \R^d$. We now assume that $q$ is positive definite as opposed to just positive semidefinite, and a member of $\cC(\R^d) \cap L^1(\R^d)$, i.e., it is continuous, bounded and integrable on $\R^d$. Then, it has a positive Fourier transform  $\hat{q} \colon \R^d \to \R^+$ which is also integrable on $\R^d$, see~\cite[Chapter~6]{W04}. We use this property to derive a regularity result for $Q$ in a general Schatten class, starting with the following lemma. 
	
	\begin{lemma}
		\label{lem:Q-sobolev}
		If there are constants $C > 0$, $\sigma > d/2$ such that 
		\begin{equation*}
		\hat{q}(\xi) \le C \left(1 + |\xi|^2 \right)^{-\sigma}
		\end{equation*}
		for all $\xi \in \R^d$,  then the operator $Q$ with kernel $q \in \cC(\bar{\cD}\times\bar{\cD})$ satisfies $Q \in \cL(H,H^{2\sigma})$.
	\end{lemma}
	\begin{proof}
		Let $v \in H = L^2(\cD)$. The function $\cD \ni x \mapsto Q v(x)$ can be extended to $\R^d$ by 
		\begin{equation*}
		Q v(x) = \int_{\cD} q(x-y) v(y) \dd y = \int_{\R^d} q(x-y) v(y) \chi_{\cD}(y) \dd y = \left(q \ast (v \chi_{\cD})\right) (x),
		\end{equation*}
		where $\chi_{\cD}(x) = 1$ for $x \in \cD$ and $0$ elsewhere.
		Since $q \in L^1(\R)$ and $v\chi_{\cD} \in L^2(\R)$, $q \ast (v \chi_{\cD}) \in L^2(\R)$, so that $\reallywidehat{q \ast (v \chi_{\cD})}$ is well-defined and \eqref{eq:fouriernormD}~implies that
		\begin{align*}
		\norm{Q v}{H^{2\sigma}(\cD)}^2 \le \norm{q \ast (v\chi_{\cD})}{H^{2\sigma}(\R^d)}^2  &= \frac{1}{(2\pi)^{\frac{d}{2}}}\int_{\R^d} |\reallywidehat{q \ast (v \chi_{\cD})}(\xi)|^2 (1 + |\xi|^2)^{2\sigma} \dd \xi \\
		&= \frac{1}{(2\pi)^{\frac{d}{2}}}\int_{\R^d} |\widehat{v \chi_{\cD}}(\xi)|^2 \hat{q}(\xi)^2 (1 + |\xi|^2)^{2\sigma} \dd \xi \\
		&\lesssim \frac{1}{(2\pi)^{\frac{d}{2}}}\int_{\R^d} |\widehat{v \chi_{\cD}}(\xi)|^2 \dd \xi = \norm{v}{H}^2,
		\end{align*}
		where we made use of Plancherel's theorem. 
	\end{proof}
	
	This lemma allows us to deduce a regularity result on $Q$, similar to Theorems~\ref{prop:Qreg} and~\ref{thm:Qnonhomhsneumann}. However, instead of just considering the estimate~\eqref{eq:Qreg}, we deduce conditions on $q$ for which the general Schatten norm condition $
	\norm{\Lambda^{r/2} Q \Lambda^{s/2}}{\cL_p(H)} < \infty, $
	with $p \in [1,\infty)$, is satisfied. Since $q$ is defined on all of $\R^d$, we cannot expect $q$ to satisfy any boundary condition in the sense of~\eqref{eq:QDirichlet} or~\eqref{eq:QNeumann}, cf.\ \cite{B05}, which, in light of~\eqref{eq:sobolev_id_1} and~\eqref{eq:sobolev_id_2}, explains the restrictive range on $r$ and $s$ below.
	
	\begin{theorem}
		\label{thm:Q-homog-schatten}
		Under the same conditions as in~Lemma~\ref{lem:Q-sobolev}, the operator $Q$ with kernel $q \in \cC(\bar{\cD}\times\bar{\cD})$ satisfies
		\begin{equation*}
		\norm{\Lambda^{\frac{r}{2}} Q \Lambda^{\frac{s}{2}}}{\cL_p(H)} < \infty
		\end{equation*}
		\begin{enumerate}[label=(\roman*)]
			\item %
			for all $r,s \in [0, 1/2)$ such that $r+s < 2\sigma-d/p$ if $\Lambda$ has Dirichlet boundary conditions and 
			\item %
			for all $r,s \in [0, 3/2)$ such that $r+s < 2\sigma-d/p$ if $\Lambda$ has Neumann boundary conditions. 
		\end{enumerate}
	\end{theorem}
	\begin{proof}
		Using Lemmata~\ref{lem:sobolevembedding} and~\ref{lem:Q-sobolev} along with~\eqref{eq:schatten-ideal}, we obtain 
		\begin{equation*}
		\norm{Q}{\cL_p(H,H^r)} < \norm{I_{H^{2\sigma} \hookrightarrow H^r}}{ \cL_p(H^{2\sigma},H^r) } \norm{Q}{\cL(H,H^{2\sigma})} < \infty
		\end{equation*}
		for $2 \sigma - r > d/p$. Lemma~\ref{lem:symmetricschatten2} along with~\eqref{eq:sobolev_id_1} and~\eqref{eq:sobolev_id_2} now complete the proof.
	\end{proof}
	
	\begin{remark}
		\label{rem:matern}
		Examples of kernels $q$ covered by the results above include the class of Mat\'ern covariance kernels \cite[Example~7.17]{LPS14}. The exponential kernel is a special case. It is given by $q(x,y) = \exp(-|x-y|)$ for $x, y \in \R$. Its Fourier transform $\hat{q}$ is, for a constant $C>0$, given by $\hat{q}(\xi) = C \left(1 + |\xi|^2 \right)^{-(d+1)/2}$. Another example of a kernel covered by the results is the Gaussian kernel  $q(x,y) = \exp(-|x-y|^2)$ with Fourier transform given by $\hat{q}(\xi) = C \exp(-\xi^2/4)$.
	\end{remark}
	
	As a special case of this theorem, we obtain conditions on $q$ that ensure the condition~\eqref{eq:tracecond} to be satisfied. We recall that this condition is given by
	\begin{equation*}
	\norm{\Lambda^{\frac{r}{2}}Q\Lambda^{\frac{r}{2}}}{\cL_1(H)} = \norm{\Lambda^{\frac{r}{2}}Q^{\frac{1}{2}}}{\cL_2(H)}^2 < \infty,
	\end{equation*} 
	with $r \ge 0$. 
	
	\begin{corollary}
		\label{cor:homtrace}
		Under the same conditions as in~Lemma~\ref{lem:Q-sobolev}, the operator $Q$ with kernel $q \in \cC(\bar{\cD}\times\bar{\cD})$ satisfies~\eqref{eq:tracecond}
		\begin{enumerate}[label=(\roman*)]
			\item for all $r \in [0, \min(\sigma - d/2,1/2))$ if $\Lambda$ has Dirichlet boundary conditions and 
			\item for all $r \in [0, \min(\sigma - d/2,3/2))$ if $\Lambda$ has Neumann boundary conditions. 
		\end{enumerate}
	\end{corollary}
	
	Recall from Section~\ref{sec:intro} that~\eqref{eq:tracecond} being satisfied is equivalent to requiring that the $H$-valued Gaussian random variable $W(t)$ with covariance $tQ$ takes values in the space $\dot{H}^r$. As such, given the ranges of $r$ and $s$ above, Corollary~\ref{cor:homtrace} can be seen as a statement on the spatial regularity (as measured in Sobolev norms) of $W(t)$ when this is regarded as a (generalized) random field in $\cD$ with stationary covariance kernel $tq$. The deduction of such regularity properties of stationary processes (i.e., when $\cD \subset \R$) based on the properties of $\hat{q}$ has a long tradition, see, e.g., \cite{CL67}. Results that deal with stationary fields on general domains with Lipschitz boundary are harder to find. One exception is~\cite{S19} which implicitly contains the Sobolev regularity result of the corollary, albeit for $\sigma \in \N$. This can be seen from the fact that under the conditions of Lemma~\ref{lem:Q-sobolev}, $Q^{1/2}(H) = H_q(\bar{\cD}) \hookrightarrow H^\sigma(\cD)$ (cf.\ \cite[Corollary~10.48]{W04}).
	
	\begin{remark}
		The result of Corollary~\ref{cor:homtrace} is sharp. Consider the setting that $d=1$ with $\cD = (0,1)$, $\Lambda = (-\Delta)$ with Neumann boundary conditions and let $q(x,y) = \exp(-|x-y|)$ for $x,y \in \cD$. Then $\sigma = 1$ and by the result above, $\norm{\Lambda^{\frac{r}{2}}Q^{\frac{1}{2}}}{\cL_2(H)} < \infty$ for all $r \in [0,1/2)$.  By \cite[Corollary~10.48]{W04}, we have $H^1 = H_q(\bar{\cD})$ with equivalent norms. Since the condition~\eqref{eq:tracecond} is equivalent to $I_{H_q(\cD) \hookrightarrow \dot{H}^r} \in \cL_2(H_q(\cD),\dot{H}^r)$, we see by~\eqref{eq:sobolev_id_2} and Lemma~\ref{lem:sobolevembedding} that $\norm{\Lambda^{\frac{r}{2}}Q^{\frac{1}{2}}}{\cL_2(H)} = \infty$ for $r \ge 1/2$.
	\end{remark}
	
	\section{Applications to SPDE approximations}
	\label{sec:spde}
	
	In this section, we reconnect to the discussion in Section~\ref{sec:intro} and highlight applications of the estimates obtained in Sections~\ref{sec:non-hom} and~\ref{sec:hom} to the numerical approximation of SPDEs on bounded domains. We list a few examples from the literature where these estimates are used as assumptions and discuss how they are used and how this relates to our results. Even though our focus is on the numerical approximation of SPDEs, the estimates we have obtained have implications also for SPDEs themselves, as seen in Section~\ref{sec:intro}. These are not restricted to stochastic reaction-diffusion equations but include other SPDEs where an elliptic operator is involved, such as stochastic wave equations and stochastic Volterra equations on bounded domains. As an example of the latter, it can be seen that a bound of type~\eqref{eq:tracecond} implies a certain regularity of the solution~\cite[Proposition~2.1]{KP14b}. 
	
	All the examples below are considered on some bounded convex domain $\cD \subset \R^d$, $d=1,2,3$.
	
	\begin{example}[Approximation of the stochastic heat equation]
		\label{ex:stochastic-heat}
		One of the most studied SPDEs from a numerical perspective is the stochastic heat equation with additive noise, given by
		\begin{equation*}
		\dd X(t) = \Delta X(t) \dd t + \dd W(t)
		\end{equation*}
		for $t \in (0,T]$, a sufficiently smooth initial value $X(0) = x \in H$ and $W$ a $Q$-Wiener process in $H = L^2(\cD)$. This can be seen as a simplified version of equations considered for the modeling of sea surface temperature and other geophysical spatio-temporal processes \cite{HH87}.
		
		In \cite{KLL12}, Dirichlet zero boundary conditions are assumed for the negative Laplacian $\Lambda=-\Delta$. Under the condition
		\begin{equation*}
		\norm{\Lambda^{\frac{r}{2}}Q}{\cL_1(H)} < \infty,
		\end{equation*}
		it is shown in \cite[Theorem~4.2]{KLL12} that a spatially semidiscrete finite element approximation $X_h$ converges weakly to $X$ in the sense that, for a smooth test functional $\phi$ on $H$,
		\begin{equation*}
		\left|\E\left[\phi(X_h(T))-\phi(X(T))\right]\right| \le C h^{2+r} |\log(h)|
		\end{equation*}
		for some constant $C > 0$ independent of $h > 0$. Here $h$ is the maximal mesh size of the finite element mesh. The range for the parameter $r$ is taken to be $[-1,\delta-1]$, where $\delta$ is the degree of the piecewise polynomials making up the finite element space in which $X_h$ is computed. 
		
		Suppose that $Q$ is an integral operator with a homogeneous kernel $q$ satisfying the conditions of Theorem~\ref{thm:Q-homog-schatten} for some $\sigma > d/2$. Suppose further that $\delta > 1$. Then, this theorem implies that $X_h$ converges weakly to $X$, essentially with rate $2+\min(2\sigma-d,1/2)$. If we did not have access to this result, we would incorrectly assume the convergence rate to be~$2$. Similar remarks hold, with potentially more dramatic differences in rates, in the non-homogeneous setting of~Theorem~\ref{thm:nonhomtracedirichlet}. The importance of sharp weak convergence rates in the application of the multilevel Monte Carlo methods for SPDE simulation has been pointed out in~\cite{L16}. 
		
		In the setting above, we required that the piecewise polynomials making up the finite element space were of degree $\delta > 1$ to see a difference in the rate for the weak error. For the strong error, on the other hand, we see an improvement in rates also when $\delta = 1$. Specifically, for the same $q$ as above, Corollary~\ref{cor:homtrace} yields that~\eqref{eq:tracecond} is fulfilled for all $r < \sigma - d/2$. As noted in~\cite{KLL12}, we have for all such $r$ the existence of a constant $C>0$ such that 
		\begin{equation*}
		\norm{X(T)-X_h(T)}{L^2(\Omega,H)} = \E\left[\norm{X(T)-X_h(T)}{H}^2\right]^{\frac{1}{2}} \le C h^{\min(r+1,2)}.
		\end{equation*}
		If we did not have access to Corollary~\ref{cor:homtrace} and only knew that $\trace(Q) < \infty$, we might incorrectly conclude that the convergence rate was $1$.
		
		Here we only mentioned the results for Dirichlet boundary conditions since these are the most frequently encountered for approximations of the stochastic heat equation. Similar remarks hold for Neumann boundary conditions, we refer to~\cite{FS91} for details on when the error estimates used in the analysis of~\cite{KLL12} hold for non-Dirichlet boundary conditions.
	\end{example}
	
	\begin{example}[Approximation of the stochastic Allen--Cahn equation]
		The stochastic Allen--Cahn equation is a non-linear version of the stochastic heat equation, given by
		\begin{equation}
		\label{eq:stochastic-allen-cahn}
		\dd X(t) = \left(\Delta X(t) + F(X(t))\right) \dd t + \dd W(t),
		\end{equation}
		in the same setting as in Example~\ref{ex:stochastic-heat}. The operator $F$ on $H$ is non-linear and given by $F(u)(x) = u(x)-u(x)^3$ for $x \in \cD$. In \cite{QW19}, a fully discrete approximation $X_{h,\Delta t}$ of $X$ is considered, based on a piecewise linear finite element discretization in space combined with a fully implicit backward Euler approximation in time. 
		
		Estimate~\eqref{eq:tracecond} with $r > 0$ is needed to even establish existence of a solution to~\eqref{eq:stochastic-allen-cahn} in~\cite{QW19} when the spatial dimension $d=3$. Moreover, also for dimensions $d=1,2$, the estimate is necessary to find optimal convergence rates of $X_{h,\Delta t}$. Specifically, under~\eqref{eq:tracecond}, \cite[Theorem~4.1]{QW19} yields the existence of a constant $C>0$ such that for all mesh sizes $h > 0$, time steps $0 < \Delta t < 1/3$ and time points $t_n = n \Delta t $ with $n \in \N$,
		\begin{equation*}
		\norm{X(t_n)-X_{h,\Delta t}(t_n)}{L^2(\Omega,H)} = \E\left[\norm{X(t_n)-X_{h,\Delta t}(t_n)}{H}^2\right]^{\frac{1}{2}} \le C \left(h^{r+1} + {\Delta t}^{\frac{r+1}{2}} \right).
		\end{equation*}
		Here $r$ is taken in the range $[-2/3,1]$. Combining our results with this, we see, for example, that if $Q$ is an integral operator with a kernel $q \in \cC^{1,1}(\bar{\cD} \times \bar{\cD})$ that satisfy the Dirichlet boundary conditions,  Theorem~\ref{thm:nonhomtracedirichlet} yields a convergence rate of order $2$ in space and $1$ in time. If $q$ is a homogeneous kernel satisfying the conditions of Corollary~\ref{cor:homtrace} for some $\sigma > d/2$, we essentially obtain convergence rates $\min(1+\sigma-d/2,3/2)$ in space and $\min(1+\sigma-d/2,3/2)/2$ in time. Without these results, we would incorrectly assume a rate of $1$ in space and $1/2$ in time. Condition~\eqref{eq:tracecond} with $r > 0$ is commonly considered for the stochastic Allen--Cahn equation, for example in \cite{CH19,KLL18}.
	\end{example}
	
	\begin{example}[Approximation of the stochastic wave equation]
		The stochastic wave equation is another SPDE commonly encountered in the literature. It is considered as a simplified model for the movement of DNA strings suspended in liquid \cite{D09}. The authors of~\cite{CLS13} analyze numerical schemes for it in the same setting as above with $H = L^2(\cD)$. It is there given
		\begin{equation}
		\label{eq:stoch-wave}
		\dd \dot{X}(t) = \Delta X(t) \dd t + \dd W(t)
		\end{equation}
		for $t \in (0,T]$. Here $\dot X$ is the time derivative of $X$, the equation is posed with two smooth initial values $X(0)$, $\dot{X}(0)$ and $\Delta$ is equipped with Dirichlet zero boundary conditions. 
		
		In~\cite{CLS13}, discretizations $X_{h,\Delta t}, \dot{X}_{h, \Delta t}$ of $X$ and $\dot X$ are obtained by a piecewise linear finite element method in space and an exponential integrator method in time. Assuming that the estimate~\eqref{eq:tracecond} holds with $r \ge 0$, \cite[Theorem~4.3]{CLS13} yields a constant $C>0$ such that for all mesh sizes $h > 0$, time steps $\Delta t > 0$ and time points $t_n = n \Delta t $ with $n \in \N$,
		\begin{equation*}
		\norm{X(t_n)-X_{h,\Delta t}(t_n)}{L^2(\Omega,H)} \le C \left(h^{\frac{2(r+1)}{3}} + {\Delta t}^{\min(r+1,1)} \right).
		\end{equation*}
		The range for $r$ is taken to be $[-1,2]$. Hence, if $Q$ is an integral operator, Theorem~\ref{thm:nonhomtracedirichlet} and Corollary~\ref{cor:homtrace} can improve the spatial convergence rate from $2/3$ (if we only knew that~\eqref{eq:tracecond} held true with $r=0$) up to $2$, under the right conditions on $q$. Furthermore, the corresponding error result in \cite[Theorem~4.3]{CLS13} for the time derivative $\dot{X}_{h,\Delta t}$ approximation requires~\eqref{eq:tracecond} to hold  with $r > 0$ to yield any convergence rate at all. 
		
		Similar remarks hold for the results of \cite{W15}, where a temporally semidiscrete exponential integrator approximation is used, and for \cite{QW17}, where a fully discrete scheme, based also on the spectral Galerkin method, is applied to a damped stochastic wave equations.
	\end{example}
	
	\begin{example}[Approximation of SPDE covariance operators]
		Recently, the authors of this paper derived error bounds for approximations of the covariance operator of solutions to SPDEs \cite{KLP21a}, using a semigroup approach. For example, in \cite[Section~3.2]{KLP21a}, fully discrete approximations $K_{h, \Delta t}(t_n)$ of covariance operators $K(t_n) = \Cov(X(t_n))$ of the solution $X$ to (a variant of) the stochastic wave equation~\eqref{eq:stoch-wave} are considered. It is shown that for $p \in \{1,2\}$, there is a constant $C>0$ such that for all $h, \Delta t \in (0,1]$
		\begin{equation}
		\label{eq:covariance-convergence}
		\norm{K(t_n)-K_{h, \Delta t}(t_n)}{\cL_p(H)} \le C \left(h^{\min(\frac{2 r}{3},2)} + {\Delta t}^{\min(\frac{2r}{3},1)} \right).
		\end{equation}
		This applies when the discretizations used are a piecewise linear finite element method with mesh size $h$ in space along with a rational approximation of the underlying semigroup with time step $\Delta t$ in time. 
		
		The convergence in~\eqref{eq:covariance-convergence} is obtained under the assumption that
		\begin{equation}
		\label{eq:Q-cov-wave} 
		\norm{Q}{\cL_p(\dot{H}^{1},\dot{H}^{r-1})} = \norm{\Lambda^\frac{r-1}{2} Q \Lambda^{-\frac{1}{2}}}{\cL_p(H)} < \infty.
		\end{equation}
		Suppose, as in the numerical simulation in~\cite[Section~3.2]{KLP21a}, that $Q$ is an integral operator with a homogeneous kernel $q$ satisfying the conditions of Theorem~\ref{thm:Q-homog-schatten} for some $\sigma > d/2$.
		In the $\cL_1(H)$ case, we obtain, using~\eqref{eq:schatten-holder} and~\eqref{eq:schatten-ideal},
		\begin{align*}
		\norm{\Lambda^{\frac{r-1}{2}} Q \Lambda^{-\frac{1}{2}}}{\cL_1(H)} &\le \norm{\Lambda^{\frac{r-1}{2}} Q}{\cL_{p_1}(H)} \norm{\Lambda^{-\frac{1}{2}}}{\cL_{p_2}(H)} \\
		&= \norm{\Lambda^{\frac{r-1}{2}} Q}{\cL_{p_1}(H)} \norm{I_{\dot{H}^1 \hookrightarrow H}}{\cL_{p_2}(\dot{H}^1,H)} \\
		&\le \norm{\Lambda^{\frac{r-1}{2}} Q}{\cL_{p_1}(H)} \norm{I_{H^1 \hookrightarrow H}}{\cL_{p_2}(H^1,H)} \norm{I_{\dot{H}^1 \hookrightarrow H^1}}{\cL(\dot{H}^1,H^1)},
		\end{align*}
		where $1/{p_1} + 1/{p_2} = 1$. In light of Lemma~\ref{lem:sobolevembedding}, we should take ${p_2} > d$. By Theorem~\ref{thm:Q-homog-schatten}, the bound is then finite for $r-1 < \min(2 \sigma - d(1 - 1/{p_2}),1/2)$ under Dirichlet boundary conditions and for $r-1 < \min(2 \sigma - d(1 - 1/{p_2}),3/2)$ under Neumann boundary conditions. By letting ${p_2}$ tend to $d$ from above and noting that $\sigma > d/2$, we see that under the condition $\hat{q}(\xi) \le C \left(1 + |\xi|^2 \right)^{-\sigma}$, $\norm{\Lambda^{(r-1)/2} Q \Lambda^{-1/2}}{\cL_1(H)} < \infty$ for all $r < 3/2$ when Dirichlet boundary conditions are used and for all $r < \min(5/2, 2 \sigma - d + 2)$ when Neumann boundary conditions are used. By an analogous argument, we obtain, under the same condition on $\hat{q}$, that $\norm{\Lambda^{(r-1)/2} Q \Lambda^{-1/2}}{\cL_2(H)} < \infty$ for all $r < 3/2$ in the Dirichlet case. In the Neumann case, this quantity is finite for all $r < \min(5/2, 2 \sigma +1)$ when $d=1$ and for all $r < 5/2$ when $d \in \{2,3\}$.
		
		As in the examples above, the use of our estimates yields higher convergence rates compared to only knowing that $Q \in \cL_1(H)$. In this case, it is important to note that if we had only used estimates on $Q^{1/2}$ as in Corollary~\ref{cor:homtrace}, we would have obtained suboptimal rates. For example, consider the case that $d=2$ and that, for a kernel $q \in \cC(\bar{\cD} \times \bar{\cD})$, there exists $\sigma \in (0,1]$ such that for a.e.\ $z \in \cD$, $q(z,\cdot) \in C^{0,\sigma}(\bar{\cD})$ with $\esssup_{z \in \cD} \norm{q(z,\cdot)}{C^{0,\sigma}(\bar{\cD})} < \infty$. In the same way as before, we have 
		\begin{equation*}
		\norm{\Lambda^{\frac{r-1}{2}} Q \Lambda^{-\frac{1}{2}}}{\cL_1(H)} \le \norm{\Lambda^{\frac{r-1}{2}} Q \Lambda^{\frac{\epsilon}{2}}}{\cL_{2}(H)}
		\norm{I_{H^{1+\epsilon}\hookrightarrow H}}{\cL_{2}(H^{1+\epsilon},H)} \norm{I_{\dot{H}^{1+\epsilon}\hookrightarrow H^{1+\epsilon}}}{\cL(\dot{H}^{1+\epsilon},H^{1+\epsilon})}
		\end{equation*}
		for sufficiently small $\epsilon > 0$.
		In the case of Neumann boundary conditions, Theorem~\ref{thm:Qnonhomhsneumann} then yields that the quantity $\norm{\Lambda^{(r-1)/2} Q \Lambda^{-1/2}}{\cL_1(H)}$ is finite for $r < 1 +\sigma$. If we only used Theorem~\ref{thm:nonhomtraceneumann}, we would instead conclude that 
		\begin{equation*}
		\norm{\Lambda^{\frac{r-1}{2}} Q \Lambda^{-\frac{1}{2}}}{\cL_1(H)} \le \norm{\Lambda^{\frac{r-1}{2}} Q^{\frac{1}{2}}}{\cL_{2}(H)}^2 \norm{I_{\dot{H}^{r}\hookrightarrow H}}{\cL(\dot{H}^{r},H)} < \infty
		\end{equation*}   
		for $r < 1 + \sigma/2$.
		Similar remarks hold for the weak convergence analysis of approximations of hyperbolic SPDEs in~\cite{KLP20, KLL12, KLL13, W15}, where the estimate~\eqref{eq:Q-cov-wave} for $p = 1$ is also assumed. 
	\end{example}
	
	\bibliographystyle{hplain}
	\bibliography{hs-regularity}	

\begin{thebibliography}{10}

\bibitem{BT04}
A.~Berlinet and C.~Thomas-Agnan.
\newblock {\em Reproducing kernel {H}ilbert spaces in probability and
  statistics}.
\newblock Kluwer Academic Publishers, Boston, MA, 2004.

\bibitem{B05}
D.~Bl\"{o}mker.
\newblock Nonhomogeneous noise and {$Q$}-{W}iener processes on bounded domains.
\newblock {\em Stoch. Anal. Appl.}, 23(2):255--273, 2005.

\bibitem{BKK20}
D.~Bolin, K.~Kirchner, and M.~Kov\'{a}cs.
\newblock Numerical solution of fractional elliptic stochastic {PDE}s with
  spatial white noise.
\newblock {\em IMA J. Numer. Anal.}, 40(2):1051--1073, 2020.

\bibitem{CFM15}
R.~Cavoretto, G.~E. Fasshauer, and M.~McCourt.
\newblock An introduction to the {H}ilbert-{S}chmidt {SVD} using iterated
  {B}rownian bridge kernels.
\newblock {\em Numer. Algorithms}, 68(2):393--422, 2015.

\bibitem{CG94}
F.~Cobos and M.~A. Garc\'{\i}a-Dav\'{\i}a.
\newblock Remarks on interpolation properties of {S}chatten classes.
\newblock {\em Bull. London Math. Soc.}, 26(5):465--471, 1994.

\bibitem{CLS13}
D.~Cohen, S.~Larsson, and M.~Sigg.
\newblock A trigonometric method for the linear stochastic wave equation.
\newblock {\em SIAM J. Numer. Anal.}, 51(1):204--222, 2013.

\bibitem{CK20}
S.~G. Cox and K.~Kirchner.
\newblock Regularity and convergence analysis in {S}obolev and {H}\"{o}lder
  spaces for generalized {W}hittle-{M}at\'{e}rn fields.
\newblock {\em Numer. Math.}, 146(4):819--873, 2020.

\bibitem{CL67}
H.~Cram\'{e}r and M.~R. Leadbetter.
\newblock {\em Stationary and related stochastic processes. {S}ample function
  properties and their applications}.
\newblock John Wiley \& Sons, Inc., New York-London-Sydney, 1967.

\bibitem{CH19}
J.~Cui and J.~Hong.
\newblock Strong and weak convergence rates of a spatial approximation for
  stochastic partial differential equation with one-sided {L}ipschitz
  coefficient.
\newblock {\em SIAM J. Numer. Anal.}, 57(4):1815--1841, 2019.

\bibitem{DPZ14}
G.~Da~Prato and J.~Zabczyk.
\newblock {\em Stochastic equations in infinite dimensions}, volume 152 of {\em
  Encyclopedia of mathematics and its applications}.
\newblock Cambridge University Press, Cambridge, second edition, 2014.

\bibitem{D09}
R.~C. Dalang.
\newblock The stochastic wave equation.
\newblock In {\em A minicourse on stochastic partial differential equations},
  volume 1962 of {\em Lecture notes in mathematics}, pages 39--71. Springer,
  Berlin, 2009.

\bibitem{DF98}
R.~C. Dalang and N.~E. Frangos.
\newblock The stochastic wave equation in two spatial dimensions.
\newblock {\em Ann. Probab.}, 26(1):187--212, 1998.

\bibitem{F51}
K.~Fan.
\newblock Maximum properties and inequalities for the eigenvalues of completely
  continuous operators.
\newblock {\em Proc. Nat. Acad. Sci. U.S.A.}, 37:760--766, 1951.

\bibitem{FS91}
H.~Fujita and T.~Suzuki.
\newblock Evolution problems.
\newblock Handbook of numerical analysis, II, pages 789 -- 928. North-Holland,
  Amsterdam, 1991.
\newblock Finite element methods. Part 1.

\bibitem{G74}
C.~Gapaillard.
\newblock Un r\'{e}sultat de compacit\'{e} pour l'interpolation de couples
  hilbertiens.
\newblock {\em C. R. Acad. Sci. Paris S\'{e}r. A}, 278:681--684, 1974.

\bibitem{G85}
P.~Grisvard.
\newblock {\em Elliptic problems in nonsmooth domains}, volume~24 of {\em
  Monographs and studies in mathematics}.
\newblock Pitman (Advanced Publishing Program), Boston, MA, 1985.

\bibitem{HH87}
K.~Herterich and K.~Hasselmann.
\newblock Extraction of mixed layer advection velocities, diffusion
  coefficients, feedback factors and atmospheric forcing parameters from the
  statistical analysis of {N}orth {P}acific {SST} anomaly fields.
\newblock {\em J. Phys. Oceanogr.}, 17(12):2145 -- 2156, 1987.

\bibitem{HE15}
T.~Hsing and R.~Eubank.
\newblock {\em Theoretical foundations of functional data analysis, with an
  introduction to linear operators}.
\newblock Wiley Series in Probability and Statistics. John Wiley \& Sons, Ltd.,
  Chichester, 2015.

\bibitem{JK09}
A.~Jentzen and P.~E. Kloeden.
\newblock The numerical approximation of stochastic partial differential
  equations.
\newblock {\em Milan J. Math.}, 77(1):205--244, 2009.

\bibitem{KZ00}
A.~Karczewska and J.~Zabczyk.
\newblock Stochastic {PDE}s with function-valued solutions.
\newblock In {\em Infinite dimensional stochastic analysis}, volume~52 of {\em
  Verhandelingen, Afdeling Natuurkunde. Eerste Reeks. Koninklijke Nederlandse
  Akademie van Wetenschappen.}, pages 197--216. R. Neth. Acad. Arts Sci.,
  Amsterdam, 2000.

\bibitem{KLP20}
M.~Kov\'{a}cs, A.~Lang, and A.~Petersson.
\newblock Weak convergence of fully discrete finite element approximations of
  semilinear hyperbolic {SPDE} with additive noise.
\newblock {\em ESAIM Math. Model. Numer. Anal.}, 54(6):2199--2227, 2020.

\bibitem{KLP21a}
M.~Kov\'{a}cs, A.~Lang, and A.~Petersson.
\newblock Approximation of {SPDE} covariance operators by finite elements: {A}
  semigroup approach.
\newblock Preprint at arXiv:2107.10109, 2021.

\bibitem{KLL12}
M.~Kov\'{a}cs, S.~Larsson, and F.~Lindgren.
\newblock Weak convergence of finite element approximations of linear
  stochastic evolution equations with additive noise.
\newblock {\em BIT}, 52(1):85--108, 2012.

\bibitem{KLL13}
M.~Kov\'{a}cs, S.~Larsson, and F.~Lindgren.
\newblock Weak convergence of finite element approximations of linear
  stochastic evolution equations with additive noise {II}. {F}ully discrete
  schemes.
\newblock {\em BIT}, 53(2):497--525, 2013.

\bibitem{KLL18}
M.~Kov\'{a}cs, S.~Larsson, and F.~Lindgren.
\newblock On the discretisation in time of the stochastic {A}llen-{C}ahn
  equation.
\newblock {\em Math. Nachr.}, 291(5-6):966--995, 2018.

\bibitem{KP14b}
M.~Kov\'{a}cs and J.~Printems.
\newblock Weak convergence of a fully discrete approximation of a linear
  stochastic evolution equation with a positive-type memory term.
\newblock {\em J. Math. Anal. Appl.}, 413(2):939--952, 2014.

\bibitem{K14}
R.~Kruse.
\newblock {\em Strong and weak approximation of semilinear stochastic evolution
  equations}, volume 2093 of {\em Lecture notes in mathematics}.
\newblock Springer, Cham, 2014.

\bibitem{L16}
A.~Lang.
\newblock A note on the importance of weak convergence rates for {SPDE}
  approximations in multilevel {M}onte {C}arlo schemes.
\newblock In R.~Cools and D.~Nuyens, editors, {\em {M}onte {C}arlo and
  quasi-{M}onte {C}arlo methods, MCQMC, Leuven, Belgium, April 2014}, volume
  163 of {\em Springer Proceedings in Mathematics \& Statistics}, pages
  489--505, 2016.

\bibitem{LP11}
A.~Lang and J.~Potthoff.
\newblock Fast simulation of {G}aussian random fields.
\newblock {\em Monte Carlo Methods Appl.}, 17(3):195--214, 2011.

\bibitem{LS15}
A.~Lang and C.~Schwab.
\newblock Isotropic {G}aussian random fields on the sphere: regularity, fast
  simulation and stochastic partial differential equations.
\newblock {\em Ann. Appl. Probab.}, 25(6):3047--3094, 2015.

\bibitem{LM72}
J.~L. Lions and E.~Magenes.
\newblock {\em Non-homogeneous boundary value problems and applications. {V}ol.
  {I}}, volume 181 of {\em Die Grundlehren der mathematischen Wissenschaften}.
\newblock Springer-Verlag, New York-Heidelberg, 1972.
\newblock Translated from the French by P. Kenneth.

\bibitem{LPS14}
G.~J. Lord, C.~E. Powell, and T.~Shardlow.
\newblock {\em An introduction to computational stochastic PDEs}.
\newblock Cambridge Texts in Applied Mathematics. Cambridge University Press,
  2014.

\bibitem{LR17}
S.~V. Lototsky and B.~L. Rozovsky.
\newblock {\em Stochastic partial differential equations}.
\newblock Universitext. Springer, Cham, 2017.

\bibitem{NPV12}
E.~D. Nezza, G.~Palatucci, and E.~Valdinoci.
\newblock Hitchhiker's guide to the fractional {S}obolev spaces.
\newblock {\em Bull. des Sci. Math.}, 136(5):521 -- 573, 2012.

\bibitem{PZ07}
S.~Peszat and J.~Zabczyk.
\newblock {\em Stochastic partial differential equations with {L}\'evy noise.
  An evolution equation approach}, volume 113 of {\em Encyclopedia of
  mathematics and its applications}.
\newblock Cambridge University Press, Cambridge, 2007.

\bibitem{PZ97}
S.~Peszat and J.~Zabczyk.
\newblock Stochastic evolution equations with a spatially homogeneous {W}iener
  process.
\newblock {\em Stochastic Process. Appl.}, 72(2):187--204, 1997.

\bibitem{QW17}
R.~Qi and X.~Wang.
\newblock An accelerated exponential time integrator for semi-linear stochastic
  strongly damped wave equation with additive noise.
\newblock {\em J. Math. Anal. Appl.}, 447(2):988--1008, 2017.

\bibitem{QW19}
R.~Qi and X.~Wang.
\newblock Optimal error estimates of {G}alerkin finite element methods for
  stochastic {A}llen-{C}ahn equation with additive noise.
\newblock {\em J. Sci. Comput.}, 80(2):1171--1194, 2019.

\bibitem{S19}
I.~Steinwart.
\newblock Convergence types and rates in generic {K}arhunen-{L}o\`eve
  expansions with applications to sample path properties.
\newblock {\em Potential Anal.}, 51(3):361--395, 2019.

\bibitem{SC08}
I.~Steinwart and A.~Christmann.
\newblock {\em Support vector machines}.
\newblock Information Science and Statistics. Springer, New York, 2008.

\bibitem{T67}
H.~Triebel.
\newblock \"{U}ber die {V}erteilung der {A}pproximationszahlen kompakter
  {O}peratoren in {S}obolev-{B}esov-{R}\"{a}umen.
\newblock {\em Invent. Math.}, 4:275--293, 1967.

\bibitem{T92}
H.~Triebel.
\newblock {\em Higher analysis}.
\newblock Hochschulb\"{u}cher f\"{u}r Mathematik. [University Books for
  Mathematics]. Johann Ambrosius Barth Verlag GmbH, Leipzig, 1992.
\newblock Translated from the German by Bernhardt Simon [Bernhard Simon] and
  revised by the author.

\bibitem{W70}
G.~Wahba.
\newblock Convergence rates of certain approximate solutions to {F}redholm
  integral equations of the first kind.
\newblock {\em J. Approximation Theory}, 7:167--185, 1973.

\bibitem{W15}
X.~Wang.
\newblock An exponential integrator scheme for time discretization of nonlinear
  stochastic wave equation.
\newblock {\em J. Sci. Comput.}, 64(1):234--263, 2015.

\bibitem{W04}
H.~Wendland.
\newblock {\em Scattered data approximation}, volume~17 of {\em Cambridge
  Monographs on Applied and Computational Mathematics}.
\newblock Cambridge University Press, Cambridge, 2005.

\bibitem{Y10}
A.~Yagi.
\newblock {\em Abstract parabolic evolution equations and their applications}.
\newblock Springer Monographs in Mathematics. Springer-Verlag, Berlin, 2010.

\bibitem{Z08}
D.-X. Zhou.
\newblock Derivative reproducing properties for kernel methods in learning
  theory.
\newblock {\em J. Comput. Appl. Math.}, 220(1-2):456--463, 2008.

\end{thebibliography}
	
\end{document}